\documentclass[12pt,a4paper]{article}

\usepackage[a4paper]{geometry}
\usepackage[utf8]{inputenx}
\usepackage{amsmath, amssymb, amsthm}
\usepackage{subfigure, graphicx}
\usepackage{stmaryrd}
\usepackage{cite}
\usepackage{units}
\usepackage{mathtools}

\usepackage[bf]{caption}
\usepackage{todonotes}

\numberwithin{equation}{section}

\theoremstyle{plain}
\newtheorem{theorem}{Theorem}[section]
\newtheorem{lemma}[theorem]{Lemma}

\newtheorem{cor}{Corollary}[theorem]

\theoremstyle{definition}
\newtheorem{definition}{Definition}[section]

\theoremstyle{remark}
\newtheorem*{remark}{Remark}

\newcommand{\nf}{{\nicefrac{+}{-}}}

\newcommand{\lJ}{\llbracket}
\newcommand{\rJ}{\rrbracket}
\newcommand{\lA}{\{}
\newcommand{\rA}{\}}

\newcommand{\cpm}{{p}}

\newcommand{\trace}{{\text{tr}}}

\newcommand{\bfx}{{x}}
\newcommand{\bfb}{{b}}

\newcommand{\bfa}{{a}}

\newcommand{\bfI}{{I}}

\newcommand{\RT}{{\mathbb{R}^3}}

\newcommand{\rgrad}{\nabla}
\newcommand{\bgrad}{\nabla_\Gamma}
\newcommand{\hgrad}{\nabla_{\Gamma_h}}

\newcommand{\blaplace}{\Delta_\Gamma}
\newcommand{\hlaplace}{\Delta_{\Gamma_h}}

\newcommand{\bn}{n}
\newcommand{\hn}{n_h}

\newcommand{\hK}{{K}}
\newcommand{\bK}{{K^{\ell}}}
\newcommand{\hE}{{E}}
\newcommand{\bE}{{E^{\ell}}}

\newcommand{\bEnK}{n_{\partial\bK}}	    
\newcommand{\bEn}{n_{E^\ell}}	    
\newcommand{\hEnK}{n_{\partial\hK}}	    
\newcommand{\hEn}{n_{E}}	

\newcommand{\bP}{P}
\newcommand{\hP}{\bP_h}

\newcommand{\tnorm}[1]{\vert\hspace{-0.3mm}\Vert#1\Vert\hspace{-0.3mm}\vert_{\Gamma_h}}
\newcommand{\bnorm}[1]{\vert\hspace{-0.3mm}\Vert#1\Vert\hspace{-0.3mm}\vert_\Gamma}

\newcommand{\proofterm}[1]{\vspace{0.5em}\noindent\emph{#1}\,}

\newcommand{\adg}{{a}} 

\newcommand{\leftgrad}{\smash[t]{\overset{\text{\scriptsize$\leftarrow$}}{\nabla}}}
\newcommand{\lbgrad}{{\leftgrad_\Gamma}}
\newcommand{\lhgrad}{{\leftgrad_{\Gamma_h}}}

\newcommand{\lproj}{\lfloor}
\newcommand{\rbproj}{\rfloor_{\bP}}
\newcommand{\rhproj}{\rfloor_{\hP}}

\begin{document}

\title{\bf A continuous/discontinuous {G}alerkin method and a priori
error estimates for the biharmonic problem on surfaces}

\author{
Karl Larsson \footnotemark[2]
    \qquad
  Mats G. Larson \footnotemark[3]
\\[4mm]\it Department of Mathematics and Mathematical Statistics, \\\it Ume{\aa} University, 
SE-901 87 Ume{\aa}, Sweden
}

\date{}


\maketitle
\renewcommand{\thefootnote}{\fnsymbol{footnote}}
\footnotetext[2]{{\tt karl.larsson@umu.se}}
\footnotetext[3]{{\tt mats.larson@umu.se}}

\begin{abstract}
\noindent
We present a continuous/discontinuous Galerkin method for approximating solutions to a fourth order elliptic PDE on a surface embedded in $\RT$.
A priori error estimates, taking both the approximation of the surface and the approximation of surface differential operators into account, are proven in a discrete energy norm and in $L^2$ norm.
This can be seen as an extension of the formalism and method originally used by Dziuk \cite{Dziuk1988} for approximating solutions to the Laplace--Beltrami problem,
and within this setting this is the first analysis of a surface finite element method formulated using higher order surface differential operators.
Using a polygonal approximation $\Gamma_h$ of an implicitly defined surface $\Gamma$ we employ continuous piecewise quadratic finite elements to approximate solutions to the biharmonic equation on $\Gamma$. Numerical examples on the sphere and on the torus confirm the convergence rate implied by our estimates.
\end{abstract}

\section{Introduction}
\subsection{Model problem and earlier work}
Numerical solutions to fourth order PDE on surfaces have several applications,
for example thin shells \cite{chapelle2003}, the Cahn--Hilliard equations \cite{Cahn1958}, or lubrication modeling \cite{lubrication}.
In this paper we for purposes of method development and analysis consider the following fourth order model problem.
Let $\Gamma$ be a smooth two-dimensional surface without boundary embedded in $\RT$.
For $f$ satisfying $\int_\Gamma f \, ds = 0$, find $u$ satisfying $\int_\Gamma u \, ds = 0$ such that
\begin{align}
\blaplace^2 u = f \quad \text{on $\Gamma$} 
\label{eq:introstrong}
\end{align}
where $\blaplace^2 u := \blaplace(\blaplace u)$ and $\blaplace$ is the Laplace--Beltrami operator.
We call this the biharmonic equation on the surface $\Gamma$.

We follow the formalism first used in \cite{Dziuk1988} for solving the Laplace--Beltrami problem
where $\Gamma$ is implicitly defined using an oriented distance function and the surface differential operators are constructed using the tangential gradient $\bgrad := \bP\rgrad$, i.e. the projection of the Cartesian gradient onto the tangential plane.
These initial results have since been extended in various ways and for various problems formulated using the second order Laplace--Beltrami operator, yielding weak formulations with terms of the form
$\int_\Gamma \bgrad u \cdot \bgrad v \, ds$
(cf. \cite{DemlowD07,Dziuk2007,Demlow2009,Olshanskii2009,Demlow2012,burman2015,hansbo2015}). By employing a second order splitting method \cite{Dziuk2007} also consider fourth order linear diffusion and the Cahn--Hilliard equation in the same framework yielding two coupled systems of equations.
In this paper we however develop a method and analysis based on a more direct approach for the fourth order surface bi-Laplacian
$\blaplace^2$. We propose and implement a continuous/discontinuous Galerkin (c/dG) method \cite{Engel2002} for the biharmonic equation on a surface and extend the analysis in \cite{Dziuk1988} to cover this method. For the second order Laplace--Beltrami problem discontinuous Galerkin methods was considered in \cite{Dedner2013,dedner2014}.

The advantages of using implicitly defined surfaces rather than global or local parameterizations are several.
As can be seen in \cite{Dziuk1988} implementation and analysis becomes fairly straightforward. The formalism is also suitable for problems where parameterization is unavailable, as may be the case in problems on evolving surfaces (cf. \cite{dziuk2007a, Olshanskii2009}). For a more thorough review of finite element methods for various surface PDE we refer to \cite{Dziuk2013}.

\subsection{Main contributions}

To our knowledge, this paper is the first paper presenting an analysis of a
surface finite element method formulated using higher order
surface differential operators, i.e. operators other than $\bgrad$, in the framework introduced in \cite{Dziuk1988}.
The formalism and tools in the framework is extended which will be valuable in
future analysis of both conforming and non-conforming finite element methods for higher order PDE on surfaces.
In particular, higher order tangential derivatives in the embedded setting are carefully defined such that they are independent of
artificial out-of-plane components in lower order derivatives. From this follows
clearly formulated surface Sobolev spaces of arbitrary order.
Further, the proof of the $L^2$ estimate requires a more refined approach, compared to the case of the Laplace--Beltrami operator, which utilizes the high regularity of the exact solution and the solution to the dual problem.

In the remainder of this section we summarize the main contributions of the paper.
For increased readability this summary is written in an informal fashion and we leave technicalities
such as formally defining operations on functions defined on different domains for Section~\ref{section:geom}.

\proofterm{Higher order differential operators on surfaces:}
As the first surface finite element method in the framework of \cite{Dziuk1988} formulated using higher order
differential operators, a number of definitions and technical results are needed for implementation
and analysis. Especially, we give the following contributions.
\begin{itemize}
\item Definitions of higher order surface differential operators (other than the Laplace--Beltrami operator) and higher order surface Sobolev spaces in a tangential calculus setting, see \S\ref{section:sobolev}.
\item Lemma~\ref{sobolevestimate} for Sobolev norm comparison between the exact and approximate surface, is a substantial extension of results in \cite[Lemma 3.1]{Dziuk1988} and \cite{Demlow2009}. In particular, the addition of estimate \eqref{aa4} greatly augment the use of this lemma when working with higher order Sobolev spaces.
\end{itemize}

\proofterm{A priori energy error estimate:}
Let $u_h$ be the finite element solution to the method: Find $u_h \in W_h$ such that $a_h(u_h,v_h)=l_h(v_h)$ for all $v_h\in W_h$, where $W_h$ is the finite element space based on continuous piecewise quadratic interpolation over $\Gamma_h$.
Here $a_h(\cdot,\cdot)$ is the symmetric bilinear form of the continuous/discontinuous Galerkin (c/dG) method formulated using approximate surface differential operators based on $\Gamma_h$.
Also, both $a_h(\cdot,\cdot)$ and the linear functional $l_h(\cdot)$ are integrated over $\Gamma_h$.
Based on the exact geometry $\Gamma$ we have the corresponding bilinear form $a(\cdot,\cdot)$
and linear functional $l(\cdot)$ for which $a(u,v)=l(v)$
for all $v \in H^4(\Gamma) + W_h$ where $u$ is the exact solution to \eqref{eq:introstrong}.
The following error estimate (see Theorem~\ref{maintheorem}) holds
\begin{align}
\bnorm{ u - u_h } \leq C h \| f \|_{L^2(\Gamma)}
\end{align}
where $C$ is a mesh independent constant and $\bnorm{\cdot}$ is the energy norm.
The proof of this estimate follows from the first Strang lemma and we derive estimates
for the following three terms
\begin{align} \label{eq:strangterms}
\bnorm{ u - u_h } &\lesssim 
\bnorm{ u - \pi u }
+
\sup_{w_h \in W_h} \frac{| a(\pi u,w_h) - a_h(\pi u,w_h) |}{\bnorm{ w_h }}
\\&\quad+
\sup_{w_h \in W_h} \frac{| l(w_h) - l_h(w_h) |}{\bnorm{ w_h }}
\nonumber
\end{align}
where $\pi:H^4(\Gamma)+W_h \rightarrow W_h$ is an interpolant. As this is a fourth order problem the first term above,
i.e. the interpolation error in energy norm, will be $\mathcal{O}(h)$ as $W_h$ is based on quadratic interpolation.
While increasing the order of geometry approximation from linear (facets) to piecewise quadratic would improve
the estimates for the two last terms in \eqref{eq:strangterms}, i.e. the quadrature error,
it would not affect the interpolation term. Thus, only increasing the order of geometry interpolation would not increase
the order of convergence in this method.

\proofterm{Non-standard analysis for the $L^2$ estimate:}
As Galerkin orthogonality doesn't hold due to geometry approximation,
i.e. for $v\in W_h$ we typically have $a(u-u_h,v) \neq 0$, 
we will get a remainder term when using a duality argument (Aubin--Nitsche's trick)
to derive an $L^2$ estimate (see Theorem~\ref{lemma:L2Gamma}).
In particular, letting $\phi\in H^4(\Gamma)$ be the solution to a dual problem,
we similarly to the Strang lemma get the following expression 
in the proof of Theorem~\ref{lemma:L2Gamma}
\begin{align}
&\| u - u_h \|_{L^2(\Gamma) / \mathbb{R}}^2
=
a(u - u_h,\phi - \pi\phi)
+
a(u - u_h,\pi\phi)
\\
&\qquad\quad=
\underbrace{a(u - u_h,\phi - \pi\phi)}_{\mathcal{O}(h^2)}
+
\underbrace{l(\pi\phi)-l_h(\pi\phi)}_{\mathcal{O}(h^2)}
+
\underbrace{a_h(u_h,\pi\phi) - a(u_h,\pi\phi)}_{???}
\label{eq:ksgjdn}
\end{align}
where the first term is $\mathcal{O}(h^2)$ due to the error and interpolation estimates in energy norm, and the second term is also $\mathcal{O}(h^2)$, see \eqref{ldiff}. However, for the third term using the same estimates
as in the energy norm estimate only gives $\mathcal{O}(h)$, see \eqref{adiff}, which is not sharp enough.
This estimate can be improved as follows.
By adding and subtracting terms, and using error and interpolation energy norm
estimates, the problem of estimating the third term is transformed into estimating
\begin{align}
a_h(u,\phi) - a(u,\phi) \label{eq:kgufng}
\end{align}
where it is important to note that $u,\phi \in H^4(\Gamma)$ by elliptic regularity. The regularity
of $u$ and $\phi$ is then utilized in the following results:
\begin{itemize}
\item Lemma~\ref{lemma:geomnonstandard} is a non-standard geometry approximation result for $\hP\cdot\bn$, the exact normal projected onto the approximate tangential plane, where we instead of using max-norm estimates prove an estimate for $(\hP\cdot\bn)\cdot\chi$
integrated over $\Gamma_h$
where $\chi$ is a sufficiently regular vector valued function. The resulting estimate is of one order higher than the max-norm estimate for $\hP\cdot\bn$.
\item In estimate \eqref{acontdiff} in Lemma~\ref{lemma:adiff} we prove that the expression in \eqref{eq:kgufng} is $\mathcal{O}(h^2)$.
For all non-zero terms in \eqref{eq:kgufng} we employ Lemma~\ref{lemma:geomnonstandard} to obtain the correct order.
\end{itemize}
An increase in geometry interpolation to continuous piecewise quadratics
would not improve the order of convergence for the $L^2$ estimate either as the first term in \eqref{eq:ksgjdn}
would still only be $\mathcal{O}(h^2)$ by the error and interpolation estimates in energy norm.
However, in this case the estimate of the third term can be done using standard techniques.
We finally remark that in the case of the Laplace--Beltrami operator the corresponding estimate of the
third term is $\mathcal{O}(h^2)$ and can be derived using standard techniques.

\subsection{Outline}

The remainder of this paper is dispositioned as follows. In \S\ref{geom} we
introduce the geometric description of the surface $\Gamma$ and define tangential
derivatives of arbitrary order which we use to define the surface differential operators.
Using the tangential derivatives we define suitable Sobolev spaces of any order on curved surfaces.
In \S\ref{section:geom} we provide assumptions and geometry approximation results for a triangulation $\Gamma_h$ of
the exact surface $\Gamma$.
By extending the domain of functions defined on either $\Gamma$ and $\Gamma_h$ to a volumetric neighborhood to $\Gamma$
we also provide comparisons of functions in surface Sobolev norms on $\Gamma$ and $\Gamma_h$.
In \S\ref{section:cdG} we begin by introducing the biharmonic problem on surfaces
and derive a broken weak formulation of the problem with bilinear form $a(\cdot,\cdot)$ and linear functional $l(\cdot)$. 
A continuous/discontinuous Galerkin method is then formulated using
surface differential operators and integration based on the approximate surface $\Gamma_h$
rendering an approximate bilinear form $a_h(\cdot,\cdot)$ and linear functional $l_h(\cdot)$. 
As approximation space we choose the space of continuous piecewise quadratic functions over $\Gamma_h$.
We derive a priori error estimates for this method in \S\ref{section:apriori}, both in a discrete energy norm and in $L^2$ norms on $\Gamma$ and $\Gamma_h$. Finally, to support our theoretical findings we in \S\ref{section:numerics} give
numerical results for two model problems with known analytical solutions.


\section{Surface geometry and differential operators}
\label{geom}

Let $\Gamma$ be a smooth two-dimensional surface without boundary embedded in $\mathbb{R}^3$. 
Assuming that $\Gamma$ is represented by an oriented distance function $d(x)$, giving positive values on the exterior of $\Gamma$, we have an outward pointing unit normal given by $\bn(x) = \rgrad d(x) \in \RT$ and extended Weingarten map given by $H(x) = \nabla^2 d \in \mathbb{R}^{3\times 3}$. 
The eigenvalues of $H$ are $\{\kappa_1(x), \kappa_2(x), 0\}$ with corresponding orthogonal eigenvectors where the eigenvector corresponding to the zero eigenvalue is in the normal direction $\bn(x)$.
For $x\in\Gamma$, $\{\kappa_i\}$ are the principal curvatures of $\Gamma$ and
thus $\trace(H)=2N$, where $N$ is the mean curvature.


We now turn to introducing the differential operators 
used to describe our equations on $\Gamma$ and we also introduce the appropriate Sobolev spaces.

\subsection{Tangential differential calculus} \label{tangentialdifferential}

As in \cite{Dziuk1988} we define the operators using tangential differential calculus \cite{delfour2000}, avoiding the need for local coordinates and Christoffel symbols. The tangential projection along $\Gamma$ is given by
$\bP = I - \bn \otimes \bn$
and we use this projection to define differential operators on the surface expressed in the global Cartesian coordinate system. 

In any of the literature cited in the introduction we have not found any explicit definition, or reference to such, of surface differential operators $D^k_\Gamma w$ for $k \geq 2$ in an embedded setting such that these only contain tangential derivatives of order $k$.
For example, as noted in \cite{delfour2000}, the Hessian $\bgrad\otimes\bgrad w$ will be unsymmetric and contain out-of-plane components.
We therefore provide an effort in this paper to explicitly define higher order surface differential
operators and surface Sobolev spaces in an embedded setting.

\subsection{Surface Sobolev spaces} \label{sect:surfsobolev}

Let $L^2(\Gamma)$ be the usual $L^2$-space on $\Gamma$ with norm
$\| T \|_{L^2(\Gamma)} := \left(\int_\Gamma T:T \, ds \right)^{1/2}$
where $T$ is a tensor (in the sense of a multi-linear array) and $:$ denotes contraction
in every tensorial dimension.
Further, let $H^m(\Gamma)$, $m\in\mathbb{N}_0^+$, be Sobolev spaces on $\Gamma$
with norm and semi-norm
\begin{align}
\|w\|_{H^m(\Gamma)}^2 := \sum_{k=0}^m | w |_{H^k(\Gamma)}^2
\qquad\text{and}\qquad
| w |_{H^k(\Gamma)} := \| D^k_{\Gamma} w \|_{L^2(\Gamma)}
\end{align}
respectively, where $D^k_{\Gamma} w$ is the tensor of order $k$ tangential derivatives defined by
\begin{align}
D^k_{\Gamma} w :=
\left\{
\begin{alignedat}{2}
& w
\qquad &&\text{for $k = 0$}
\\
&\left\lproj \left( D^{k-1}_{\Gamma} w \right) \otimes \leftgrad \right\rbproj
\qquad &&\text{for $k \geq 1$}
\end{alignedat}\right.
\label{eq:Dk}
\end{align}
Here the arrow over the gradient indicates that $\leftgrad$ operates to the left and $\left\lproj \cdot \right\rbproj$ denotes the projection onto the tangent space in each tensorial dimension.
To express this projection more formally we use the $n$-mode product denoted $\times_n$  which for a $k$:th order tensor $T \in \mathbb{R}^{3\times \cdots \times 3}$ and a matrix $A\in\mathbb{R}^{3\times 3}$ componentwise is defined as
\begin{align}
\left(T \times_n A\right)_{i_1 i_2 \cdots i_k}
= \sum_{j = 1}^3 T_{i_1 \cdots i_{n-1} j i_{n+1} \cdots i_k} A_{i_n j}
\qquad
\text{for $i_1,\cdots,i_k \in \{1,2,3\}$}
\label{eq:nmodematrix}
\end{align}
see eg. \cite[Sect. 2.5]{tensor2009}. For example, for two matrices $A,B\in\mathbb{R}^{3\times 3}$, $AB = B \times_1 A = A \times_2 B^T$.
In this notation the projection of a $k$:th order tensor $T$ is written $\left\lproj T \right\rbproj  = T \times_1 P \times_2 P ... \times_{k} P$ and we remark that this expression is independent of the order in which the $n$-mode products are evaluated.
Note that the projection will ensure that no out-of-plane components exist and in turn the inductive definition
\eqref{eq:Dk} ensures that higher order derivatives are not affected by artificial out-of-plane components (derivatives of out-of-plane components may of course be tangential).
We will also use the following notation for tangential derivatives of a tensor $T$
\begin{align}
D_\Gamma T := \left\lproj T \otimes \leftgrad \right\rbproj
\label{eq:tensorderivative}
\end{align}
and we note that $D_\Gamma^k w = \underbrace{D_\Gamma D_\Gamma \cdots D_\Gamma}_{\text{$k$ copies}} w$.

In terms of the tangential gradient $\bgrad := \bP \nabla$ we may explicitly state the first two tensors
of tangential derivatives
\begin{align}
D^1_\Gamma w &= \left\lproj w  \otimes \leftgrad \right\rbproj
=
\left\lproj \rgrad w \right\rbproj
=
(\rgrad w) \times_1 P
= P\nabla w = \bgrad w
\\
D^2_{\Gamma} w &= \left\lproj (\bgrad w)  \otimes \leftgrad \right\rbproj
=
\bigl( (\bgrad w) \otimes \leftgrad \bigr) \times_1 P \times_2 P
\\&\quad\nonumber
=
\Bigl(\bigl( (\bgrad w) \otimes \leftgrad \bigr)P^T \Bigr) \times_1 P
= \bP \Bigl( (\bgrad w) \otimes \lbgrad \Bigr)
\end{align}
and we note that the Laplace--Beltrami operator is
$\blaplace w := \bgrad \cdot \bgrad w = \trace(D^2_{\Gamma} w)$.

\section{Geometry approximation} \label{section:geom}

We define a volumetric neighborhood $U$ to $\Gamma$ by
$
U = \{ x \in \mathbb{R}^3 \ | \ \text{dist}(x,\Gamma) < \delta \}
$
where $\delta$ is small enough such that the
closest point mapping $\cpm:U\rightarrow \Gamma$ defined by
\begin{align}
\cpm(x) = x - d(x)\bn(\cpm(x))
\label{eq:cpm}
\end{align}
is unique and there exists a constant $C$ such that
\begin{align}
\| D^\alpha d \|_{L^\infty(U)} \leq C \qquad \text{for $0 \leq |\alpha| \leq k+1$}
\label{dcrit}
\end{align}
where $\alpha$ is a multi-index, $\{ D^\alpha d \}$ is the set of all Cartesian partial derivatives of order $|\alpha|$ and $\left\| \cdot \right\|_{L^\infty(U)} := \sup_{x\in U} | \cdot |$.
The minimum value of $k$ in \eqref{dcrit} is determined by the highest order of Sobolev norm approximation needed when using Lemma~\ref{sobolevestimate}, so for the purposes of the analysis in this paper we assume $k=3$.

Throughout this paper we will use assumptions and approximation results from
\cite{Dziuk1988,DemlowD07,Demlow2009}
which we present in this section.
While we try to provide complete proofs for the approximation results we
especially recommend reviewing \cite{Demlow2009}
for more general results.

\subsection{Approximate surface $\Gamma_h$ and partitioning of $\Gamma$} \label{geomapprox}
As an approximation to $\Gamma$ we consider a discrete polygonal surface $\Gamma_h \subset U$ with triangular faces whose vertices lie on $\Gamma$.
Further, let the triangle faces be shape regular and quasi-uniform of diameter $h$, cf. \cite{LarsonBengzon}, and denote the set of triangle faces $\mathcal{K}=\{ K \}$. Let $\mathcal{E}=\{ E \}$ be the set of edges in $\mathcal{K}$.
The face normal on each face $K$ is denoted by $\hn$ and the conormal to $\hK$ is denoted by $\hEnK$. 
Thus, the projection onto the tangent space of the approximate surface is given by $\hP=\bfI - \hn\otimes\hn$.
Further, on the exact surface $\Gamma$ we let $\mathcal{K}$ and $\mathcal{E}$ implicitly define a partitioning
through the closest point mapping \eqref{eq:cpm} such that the curved triangles are given by
$\bK = \{ \cpm(\bfx) : \bfx\in\hK \}$
and the curved edges between the curved triangles are given by
$\bE = \{ \cpm(\bfx) : \bfx\in\hE \}$.
To denote the domain consisting of all triangle edges on $\Gamma_h$ respectively on $\Gamma$ we
use the notations
\begin{align}
\mathcal{E}_h := \bigcup_{E \in \mathcal{E}} E
\qquad\text{and}\qquad
\mathcal{E}_\Gamma := \bigcup_{E \in \mathcal{E}} E^\ell
\end{align}
We denote the conormal to the curved triangle $\bK$ by $\bEnK$.
An illustration of a curved triangle with its facet approximation and their respective conormals is given in Figure~\ref{fig:conormals}.

For each edge $E \in \mathcal{E}$ between two neighboring triangles we name one triangle $\hK_+$ and the other $K_-$.
On edges we denote the conormals to these triangles, i.e. the outward pointing normals to $\partial\hK_\nf$, by $\hEnK^+$ and $\hEnK^-$, respectively. Analogously, we on each curved edge $\bE$ between curved triangles $\bK_\nf$ denote the conormals by $\bEnK^\nf$. Note that $\bEnK^+ + \bEnK^- = 0$ as $\left(\bK_+ \cup \bK_-\right) \subset \Gamma$
which is smooth.

\begin{figure}
\centering
\includegraphics{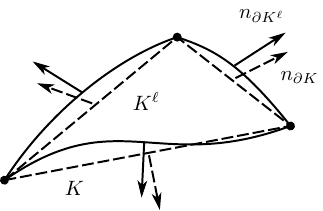}
\caption{Illustration of facet triangle $K$, using dashed lines, and lifted triangle $K^\ell \subset \Gamma$,
using solid lines, with their respective conormals.}
\label{fig:conormals}
\end{figure}

\subsection{Mapping between $\Gamma$ and $\Gamma_h$}
To map functions between the approximate and exact surfaces we extend functions to $U$
such that they are constant in the normal direction $\bn$. We denote extended functions by superscript $\ell$.
More formally, for a function $w$ defined on the exact surface $\Gamma$ we define the extension to $U$ by
\begin{align}
w^\ell(x) =  w \circ \cpm \qquad \text{for $x\in U$} \label{eq:extension1}
\end{align}

For a function $w$ defined on the approximate surface $\Gamma_h$ we first define its lifting onto $\Gamma$ by
\begin{align}
\widetilde{w}(x) = w \circ x_h  \label{explicitextension}
\end{align}
where $x_h(x)$ is the unique solution to $p(x_h)=x$ for $x_h\in\Gamma_h$, i.e. the solution to
\begin{align}
x_h = x + d(x_h) \bn(x)  \qquad \text{where $x_h \in \Gamma_h$ and $x \in \Gamma$}
\end{align}
Analogously to \eqref{eq:extension1}, we then for a function $w$ defined on $\Gamma_h$ define the extension to $U$ by
\begin{align}
w^\ell(x) =  \widetilde{w} \circ \cpm  \qquad \text{for $x\in U$} \label{eq:extension2}
\end{align}

Throughout the remainder of the paper we for clarity of notation let it be implied that functions defined on $\Gamma$ and $\Gamma_h$ are extended to $U$ by \eqref{eq:extension1} and \eqref{eq:extension2}, respectively, and only use the superscript $\ell$ notation when necessary or for emphasis.

\subsection{Geometry approximation results} \label{section:geomapprox}
In the following lemma we collect a number of approximation results for quantities defined by $\Gamma$ and $\Gamma_h$.
Note that even though we here do not explicitly denote geometrical quantities on $\Gamma$ by superscript $\ell$
these are assumed extended to $U$ by \eqref{eq:extension1}.
Further, in this lemma and throughout the paper we will for inequalities use the notation $a \lesssim b$
by which we mean that there exists a constant $c$ independent of the mesh size parameter $h$
such that $a \leq c b$.
\begin{lemma}[Geometry approximation] \label{lemma:normalinterpolation}
Let $\{ \Gamma_h \}$ be a family of polygonal approximations to $\Gamma$
with $\Gamma_h \in U$ and mesh size parameter $0 <h < h_0$.
For sufficiently small $h_0$ the following estimates hold
\begin{align}
\| d \|_{L^\infty(\Gamma_h)} &\lesssim h^2 \label{dest} \\
\| \bn - \hn \|_{L^\infty(\Gamma_h)} &\lesssim  h \label{nest} \\
\| \bP \cdot \hn \|_{L^\infty(\Gamma_h)} &\lesssim  h \label{Pnhest} \\
\| \hP \cdot \bn \|_{L^\infty(\Gamma_h)} &\lesssim  h \label{Phnest} \\
\| 1- \bn\cdot\hn \|_{L^\infty(\Gamma_h)} &\lesssim  h^2 \label{oneest} \\
\| \bEnK^\nf - \bP \hEnK^\nf \|_{L^\infty(\mathcal{E}_h)} &\lesssim  h^2 \label{test}
\end{align}
with constants depending on derivatives of $d$.
\end{lemma}
\begin{proof}
We prove this lemma in Appendix \ref{geomproof}.
\end{proof}

Let $ds$ and $ds_h$ be the surface measures of $\Gamma$ and $\Gamma_h$, respectively. For $\bfx\in\Gamma_h$ we let $\mu_h$ satisfy
$
\mu_h(\bfx) ds_h(\bfx) = ds \circ \cpm(\bfx)
$
and by results in \cite{DemlowD07,Demlow2009} we have
\begin{align}
\mu_h(x) = (\bn\cdot\hn)(1 - d(\bfx)\kappa_1(x))(1 - d(\bfx)\kappa_2(\bfx)) \label{integralmeasure2}
\end{align}
where
\begin{align}
\kappa_i(\bfx)=\frac{\kappa_i(\cpm(\bfx))}{1 + d(\bfx)\kappa_i(\cpm(\bfx))}
\label{eq:kappaextended}
\end{align}
Using \eqref{integralmeasure2}, \eqref{dest}, and \eqref{oneest} yields the estimate
\begin{align}
\| 1 - \mu_h \|_{L^\infty(\Gamma_h)} \lesssim h^2 \label{muest}
\end{align}

Further, we will need the following non-standard geometry approximation result for an integrated quantity. Note that the $L^1(\Gamma_h)$ and $W_1^1(\Gamma_h)$ norms are defined by
\begin{align}
\| \cdot \|_{L^1(\Gamma_h)} := \sum_{K\in\mathcal{K}} \int_{\hK} \| \cdot \| \, ds
\qquad
\| \cdot \|_{W_1^1(\Gamma_h)} := \| \cdot \|_{L^1(\Gamma_h)} + \| \cdot \otimes \leftgrad \|_{L^1(\Gamma_h)}
\label{eq:L1W1norms}
\end{align}
where $\| T \| := \sqrt{T:T}$ for a tensor $T$, i.e. $\| \cdot \|$ is the absolute value for a scalar, the Euclidean norm for a vector and the Frobenius norm for a matrix.
\begin{lemma}[$P_h n$ lemma] \label{lemma:geomnonstandard}
For $\chi\in [W_1^1(\Gamma_h)]^3$ it holds
\begin{align}
\left| \int_{\Gamma_h} (\hP \cdot \bn) \cdot \chi \, ds \right| \lesssim  h^2 \| \chi \|_{W_1^1(\Gamma_h)}
\end{align}
where $\{\Gamma_h\}$ fulfills the requirements of Lemma \ref{lemma:normalinterpolation}.
\end{lemma}
\begin{proof}
Using Green's formula elementwise we obtain the identity
\begin{align}
\int_{\Gamma_h} ( \hP \bn ) \cdot \chi \, ds
&=
\int_{\Gamma_h} ( \hP \nabla d) \cdot \chi \, ds
\\&=
-\int_{\Gamma_h} d \, \nabla\cdot(\hP \chi) \, ds
+ \int_{\mathcal{E}_h} d \, \lJ \hEnK \rJ \cdot \chi \, dl
= I + II
\end{align}
where $\lJ \hEnK \rJ = \hEnK^+ + \hEnK^-$.

\proofterm{Term I:}
We have estimates
\begin{align}
\left| I \right| &\lesssim \| d \|_{L^\infty(\Gamma_h)}
\left(
\| \nabla\cdot\hn \|_{L^\infty(\Gamma_h)}
\| \hn\cdot\chi \|_{L^1(\Gamma_h)}
+
\| \nabla\cdot\chi \|_{L^1(\Gamma_h)}
\right)
\\&\lesssim
h^2 \| \chi \|_{W_1^1(\Gamma_h)}
\end{align}
where we used Hölder's inequality and the bound \eqref{dest} for $d$.

\proofterm{Term II:}
For the second term we have
\begin{align}
\left| II \right| &:=
\left| \int_{\mathcal{E}_h} d \left\lJ \hEnK \right\rJ \cdot \chi \, dl \right|
\\&\lesssim
\| d \|_{L^\infty(\mathcal{E}_h)}
\| \lJ \hEnK \rJ \|_{L^\infty(\mathcal{E}_h)}
\| \chi \|_{L^1(\mathcal{E}_h)}
\\&\lesssim
\| d \|_{L^\infty(\Gamma_h)}
\| \lJ \hEnK \rJ \|_{L^\infty(\mathcal{E}_h)}
\left(
h^{-1}
\| \chi \|_{L^1(\Gamma_h)}
+
\| \chi \otimes \lhgrad \|_{L^1(\Gamma_h)}
\right)
\\&\lesssim h^2
\left(
\| \chi \|_{L^1(\Gamma_h)}
+
h
\| \chi \otimes \lhgrad \|_{L^1(\Gamma_h)}
\right)
\\&\lesssim
h^2 \| \chi \|_{W_1^1(\Gamma_h)}
\end{align}
where we use Hölder's inequality, a trace inequality, the bound \eqref{dest} for $d$
and the following estimate for the jump in the conormal
$\| \lJ \hEnK \rJ \|_{L^\infty(\mathcal{E}_h)}
\leq \| \hEnK^+ - \bEnK^+ \|_{L^\infty(\mathcal{E}_h)}
+ \| \hEnK^- - \bEnK^- \|_{L^\infty(\mathcal{E}_h)} \lesssim h$.
\end{proof}

\begin{remark}
The foundation of this proof is independent of the order of geometry approximation. By using bounds on $d$ for higher order geometry approximations (see eg. \cite{Demlow2009}) we yield
an estimate on the form
\begin{align}
\left| \int_{\Gamma_h^k} (\hP \cdot \bn) \cdot \chi \, ds \right| \lesssim  h^{1+k} \| \chi \|_{W_1^1(\Gamma_h^k)}
\end{align}
where $k$ is the polynomial order of the continuous piecewise polynomial geometry approximation $\Gamma_h^k$.
\end{remark}

\subsection{Sobolev norm approximation} \label{section:sobolev}

As $\Gamma_h$ is piecewise smooth we on each triangle face define tensors of tangential derivatives analogously to \eqref{eq:Dk}, i.e.
\begin{align}
D^k_{\Gamma_h} w :=
\left\{
\begin{alignedat}{2}
& w
\qquad &&\text{for $k = 0$}
\\
&\left\lproj \left( D^{k-1}_{\Gamma_h} w \right) \otimes \leftgrad \right\rhproj
\qquad &&\text{for $k \geq 1$}
\end{alignedat}\right.
\label{eq:Dhk}
\end{align}
and we denote the approximate surface differential operators by subscript $\Gamma_h$, for example $\hgrad$
and $\hlaplace$.

As a consequence of $\Gamma_h$ being only piecewise smooth the natural Sobolev spaces
on $\Gamma_h$ are broken, which we indicate on Sobolev spaces by subscript $h$, and
we introduce the following semi-norms for the broken Sobolev spaces
on the approximate and exact surfaces
\begin{align}
\left| v \right|_{H_h^{k}(\Gamma_h)}
:=
\left(
\sum_{K \in\mathcal{K}}
| v |_{H^k(\hK)}^2
\right)^{\frac{1}{2}}
\qquad\quad
\left| v \right|_{H_h^{k}(\Gamma)}
:=
\left(
\sum_{K \in\mathcal{K}}
| v |_{H^k(\bK)}^2
\right)^{\frac{1}{2}}
\end{align}

To compare functions in these norms on the exact and approximate surfaces we now present the following results from \cite{Dziuk1988,Demlow2009},
which we extend with estimate \eqref{aa4}
and
prove in Appendix \ref{sobolevproof}.
\begin{lemma}[Sobolev norm equivalence] \label{sobolevestimate}
Let $v \in \{ w^\ell : w \in H^k_h(\Gamma) \cap C^0(\Gamma)\}$ or $v \in \{ w^\ell : w\in H^k_h(\Gamma_h) \cap C^0(\Gamma_h)\}$ for an integer $k \geq 1$
and mesh size parameter $0<h<h_0$. For sufficiently small $h_0$ the following inequalities hold
\begin{alignat}{2}
\| v \|_{L^2(\Gamma)} & \lesssim & \ \| v \|_{L^2(\Gamma_h)} &\lesssim \| v \|_{L^2(\Gamma)} \label{aa1}
\\
| v |_{H^1(\Gamma)} & \lesssim & | v |_{H^1(\Gamma_h)} \, &\lesssim | v |_{H^1(\Gamma)} \label{aa2}
\\
& &| v |_{H^k_h(\Gamma_h)} &\lesssim \sum_{m=1}^k | v |_{H^m_h(\Gamma)} \label{aa3}
\\
| v |_{H^k_h(\Gamma)} & \lesssim \mathrlap{\sum_{m=1}^k | v |_{H^m_h(\Gamma_h)}} & & \label{aa4}
\end{alignat}
with constants depending on derivatives of the distance function $d$.
\end{lemma}
\begin{proof}
We prove this lemma in Appendix \ref{sobolevproof}.
\end{proof}
\begin{remark}
A consequence of this lemma is that if $v\in H^k(\Gamma)$ then ${v^\ell}|_{\Gamma_h} \in H^k_h(\Gamma_h)$,
and likewise if $v_h \in H^k_h(\Gamma_h)$ then ${v_h^\ell}|_\Gamma \in H^k_h(\Gamma)$.
\end{remark}

\section{The continuous/discontinuous Galerkin method} \label{section:cdG}
In this section we present the biharmonic problem on a surface $\Gamma$ and derive a weak formulation of the problem
suitable for our purposes. To deal with the $H^2(\Gamma)$ conformity requirement of the biharmonic problem
we present a continuous/discontinuous Galerkin (c/dG) method \cite{Engel2002}. The method is formulated on the approximate surface $\Gamma_h$
with surface differential operators also based on $\Gamma_h$.

\subsection{Notation}
For elements $f,g$ in an inner product space with inner product $\left\langle \cdot,\cdot \right\rangle$ we
let $(f,g)_\Omega$ denote integration over $\Omega$ such that
$(f,g)_\Omega := \int_{\Omega} \left\langle f,g \right\rangle \, ds$.
On an edge $E=\partial K_+ \cap \partial K_-$ we define the jump and average by
\begin{align}
\lJ g \rJ &:= \lim_{c\rightarrow 0+} \left(g(\bfx-c\hEnK^+) - g(\bfx-c\hEnK^-)\right)
\\
\lA g \rA &:= \lim_{c\rightarrow 0+} \frac{1}{2}
\left( g(\bfx-c\hEnK^+) + g(\bfx-c\hEnK^-) \right)
\end{align}
respectively,
and on the exact surface we define the jump and average over $E^\ell$ analogously using the conormals $\bEnK^\nf$.

\subsection{The biharmonic problem}
We consider the following model problem: Given $f\in L^2(\Gamma)$ with $(f,1)_\Gamma=0$, find $u$ such that
\begin{alignat}{2}
\blaplace^2 u &= f  &\qquad&\text{on $\Gamma$} \label{biharmonic} \\
(u,1)_\Gamma &= 0 \label{averagesolution}
\end{alignat}
where $\blaplace^2 u:=\blaplace(\blaplace u)$.
As we consider surfaces $\Gamma$ without boundary, i.e. closed manifolds, we include the criterion \eqref{averagesolution} to make the problem well posed. This is more easily seen for \eqref{biharmonic} in weak form: Given $f\in L^2(\Gamma)$, find $u\in H^2(\Gamma)$ such that
\begin{alignat}{2}
(\blaplace u, \blaplace v)_\Gamma &= (f,v)_\Gamma  &\qquad&\text{for all $v \in H^2(\Gamma)$}
\label{eq:confbiharmweak}
\end{alignat}
The nullspace of $\blaplace$ on a closed manifold is the space of constant functions,
and then by \eqref{averagesolution} the only function in the nullspace is the zero function.

For smooth surfaces $\Gamma$ without boundary we have the following elliptic regularity estimate
\begin{align}
\| u \|_{H^4(\Gamma)} \lesssim \| f \|_{L^2(\Gamma)} \label{eq:stabest}
\end{align}
under the restriction of \eqref{averagesolution}, see \cite[Th. 27]{besse1987}.

\subsection{Green's formula on curved surfaces}
Consider a smooth surface $\Sigma$ with piecewise smooth boundary $\partial\Sigma$ and surface normal $\bn$.
For functions $v :\mathbb{R}^3 \rightarrow \mathbb{R}^3$ and $w:\mathbb{R}^3 \rightarrow \mathbb{R}$
the Green's formula on $\Sigma$ reads
$
\left( \rgrad \cdot v , w \right)_{\Sigma} = \left( n_{\partial\Sigma} \cdot v , w \right)_{\partial\Sigma} - \left( v , \rgrad w \right)_{\Sigma}
$
where $n_{\partial\Sigma}$ is the outward pointing normal to $\partial\Sigma$. Using the definition of the tangential gradient we may instead write a Green's formula with tangential operators
\begin{align}
\left( \nabla_\Sigma \cdot v , w \right)_{\Sigma} = \left( n_{\partial\Sigma} \cdot v , w \right)_{\partial\Sigma}
- \left( v , \nabla_\Sigma w \right)_{\Sigma}
+ \left( \trace(H)\bn \cdot v , w \right)_{\Sigma}
\label{eq:greens}
\end{align}
where we note that we get an additional term which includes the mean curvature of the surface.
In the next section we will however notice that for the weak formulation of the biharmonic problem on a curved surface all curvature terms vanish as the vector $v$ will always be a tangent vector and thus $\bn\cdot v = 0$.

\subsection{Broken weak formulation} \label{section:weak}
The requirement on a conformal method based on \eqref{eq:confbiharmweak} in practice means
defining an approximation space which is $C^1(\Gamma)$. Due to the
intricacies involved of defining such approximation spaces we instead aim for a
continuous/discontinuous Galerkin method where the approximation space rather is in the
broken space $H^4_h(\Gamma) \cap C^0(\Gamma)$.
Multiplying the biharmonic equation on a curved surface \eqref{biharmonic} by $v \in H_h^4(\Gamma)\cap C^0(\Gamma)$,
integrating over $\Gamma$ and applying Green's formula two times gives
\begin{align}
(f,v)_{\Gamma}
&= \left( \blaplace^2 u , v \right)_{\Gamma}
=
- \left( \bgrad \blaplace u , \bgrad v \right)_{\Gamma}
=
- \sum_{K\in\mathcal{K}}\left( \bgrad \blaplace u , \bgrad v \right)_{\bK}
\\&=
\sum_{K\in\mathcal{K}}
\big(
\left( \blaplace u , \blaplace v \right)_{\bK}
-
\left(  \blaplace u  , \bEnK \cdot \bgrad v \right)_{\partial \bK}
\big)
\end{align}
where $\bEnK$ is the conormal to $\bK$, i.e. the outward pointing normal to ${\partial \bK}$,
and the curvature terms in the Green's formula \eqref{eq:greens} vanish as $\bn\cdot\bgrad = 0$.
Introducing the notation $\bEn = \lJ \bEnK \rJ /2 = \bEnK^+$ and summing over $\mathcal{K}$ we get the weak formulation: Find $u \in H^4(\Gamma)$ satisfying $(u,1)_\Gamma = 0$ such that
\begin{align} \label{confsumweak}
(f,v)_{\Gamma}
&=
\sum_{K \in \mathcal{K}} (\blaplace u, \blaplace v )_\bK
-
\sum_{E \in \mathcal{E}}
\left(\blaplace u , \bEn \cdot \lJ \bgrad v \rJ \right)_\bE
\end{align}
for all $v\in H_h^4(\Gamma) \cap C^0(\Gamma)$.
For a function $u \in H^4_h(\Gamma) \cap C^0(\Gamma)$ the term $\blaplace u$ is undefined on interior edges $\bE$ and we therefore extend \eqref{confsumweak} by defining
\begin{alignat}{2}
\blaplace u &:= \lA \blaplace u \rA - \beta h^{-1} \bEn \cdot \lJ\bgrad u \rJ 
\qquad &&\text{on $\bE$}
\end{alignat}
where $\beta$ is a positive parameter needed to achieve stability for the method, see Lemma~\ref{lemma:basics}.
To make the bilinear form symmetric we also add the term
\begin{align}
-
\sum_{E \in \mathcal{E}} \left(\bEn \cdot \lJ \bgrad u \rJ , \lA \blaplace v \rA \right)_\bE
\end{align}
Note that the above modifications does not affect the consistency of the method as $\lA \blaplace u \rA = \blaplace u$ and $\bEn \cdot \lJ \bgrad u \rJ =0$ on $\bE$ for $u\in H^4(\Gamma)$
due to the Sobolev embedding $H^4(\Gamma) \hookrightarrow C^2(\Gamma)$, see \cite[Thm. 2.20]{aubin1982}.
To allow for a more abstract presentation we let the bilinear form $a(\cdot,\cdot)$ be given by
\begin{multline}
\adg(u,v) := 
\sum_{K \in \mathcal{K}} (\blaplace u, \blaplace v )_\bK
\\-
\sum_{E \in \mathcal{E}} \Big( \left(\left\lA \blaplace u \right\rA, \bEn \cdot \lJ \bgrad v \rJ )_\bE
+ ( \bEn \cdot \lJ \bgrad u \rJ , \left\lA \blaplace v \right\rA \right)_\bE \Big) \\
+ \sum_{E \in \mathcal{E}} \beta \left( h^{-1} \bEn \cdot \lJ \bgrad u \rJ, \bEn \cdot \lJ \bgrad v \rJ \right)_\bE
  \label{thea}
\end{multline}
and linear functional $l(\cdot)$ be given by
\begin{align}
l(v):=(f,v)_\Gamma
\end{align}
We now introduce the following function spaces extended to $U$ using \eqref{eq:extension1}
\begin{align}
V
:=
\left\{ \ w^\ell \ : \ w \in
H^4(\Gamma)
\ \right\}
\qquad\quad
W
:=
\left\{ \ w^\ell \ : \ w \in
H_h^4(\Gamma)\cap C^0(\Gamma)
\ \right\}
\end{align}
where we note that $V \subset W$, and that $a(u,v)$ and $l(v)$ clearly are defined for $u,v \in W$.
The weak formulation of our continuous problem thus reads: Find $u\in V$ satisfying $(u,1)_\Gamma=0$ such that
\begin{align}
\adg(u,v)=l(v)  \qquad \text{for all $v\in W$} \label{extendedweak}
\end{align}

Both $a(\cdot,\cdot)$ and $l(\cdot)$ are formulated using the exact surface and the exact differential
operators, i.e. using information that in practice may be unavailable.
It would thus be impractical to directly formulate our method based on \eqref{extendedweak}
and we therefore in the next section formulate our method using approximations to $a(\cdot,\cdot)$ and $l(\cdot)$ based on $\Gamma_h$.

\subsection{The continuous/discontinuous Galerkin method}
On each facet edge $E$ we define an approximation to the element boundary normal $\bEn$ by
\begin{align} \label{htE}
\hEn := \frac{\hEnK^+ - \hEnK^-}{ 1 - \hEnK^+ \cdot \hEnK^-}
\end{align}
and note that this definition
has the property
\begin{align} \label{eq:tproperty}
\hEn \cdot \lJ\hgrad w\rJ = \hEnK^+ \cdot \hgrad w^+
+ \hEnK^- \cdot \hgrad w^-
\end{align}
which will simplify the analysis.
In \cite{Dedner2013} numerical experiments using variations for the definition of $\hEn$ in
a dG method for the Laplace--Beltrami problem yield the conclusion that \eqref{htE} is preferred.

By simply replacing the various terms by its discrete analogs the resulting approximate bilinear form $a_h(\cdot ,\cdot )$ on the discrete surface $\Gamma_h$ reads
\begin{multline}
a_h(u_h,v) := 
\sum_{K \in \mathcal{K}} ( \hlaplace u_h , \hlaplace v )_\hK
\\-
\sum_{E \in \mathcal{E}} \bigg( \left(\left\lA \hlaplace u_h \right\rA, \hEn \cdot \lJ \hgrad v \rJ \right)_E
+ \left( \hEn \cdot \lJ \hgrad u_h \rJ, \left\lA \hlaplace v \right\rA \right)_\hE \bigg)
\\
+ \sum_{E \in \mathcal{E}}
\beta \left( h^{-1} \hEn \cdot \lJ \hgrad u_h \rJ, \hEn \cdot \lJ\hgrad v\rJ \right)_\hE
  \label{theah}
\end{multline}
and the approximate linear functional $l_h(\cdot)$ is given by
\begin{align}
l_h(v) := (f_h,v)_{\Gamma_h}
\end{align}
where we as in \cite{Dziuk1988} define
\begin{align}
f_h := f - |\Gamma_h|^{-1}(f,1)_{\Gamma_h} \label{eq:fh}
\end{align}
and note that $(f_h,1)_{\Gamma_h}=0$. Clearly, both $a_h(u,v)$ and $l_h(v)$ are defined for functions $u,v \in W$.

We choose our finite element space $W_h$ as the space of continuous, piecewise quadratic polynomials on 
the approximate surface $\Gamma_h$, which we after extending to $U$ via \eqref{eq:extension2} express as
\begin{align}
W_h
:=
\left\{ \
w^\ell \in C^0(U) \ : \ w \in
\bigoplus_{K \in \mathcal{K}} \mathcal{P}_2(K)
\ \right\}
\end{align}
and note that $W_h \subset W$ by Lemma~\ref{sobolevestimate}, but $W_h \not\subset V$.
We now formulate our continuous/discontinuous Galerkin method: Find $u_h \in W_h$ such that
\begin{align}
a_h(u_h,v)&=l_h(v) \qquad \text{for all $v\in W_h$} \label{approxweak}
\\
(u_h,1)_{\Gamma_h} &= 0 \label{approxcrit}
\end{align}

In the next section we turn to the theoretical analysis of this method.
\section{A priori error estimates} \label{section:apriori}

We will now prove error estimates for our method in energy and $L^2$ norms using
assumptions and approximation properties presented and proved in \S\ref{section:geom}.
We begin by defining the norms we will work with in the next section. Then we
establish a number of preliminary lemmas in \S\ref{section:preliminaries} which we use
in the proofs of our main theorems in \S\ref{section:apriorierrorest}.

\subsection{Energy norms and definitions} \label{section:energynorm}
We equip $W$ with the following discrete energy norm.
\begin{align} \label{hnorm}
\tnorm{w}^2
:=
\sum_{K \in \mathcal{K}}
\| \hlaplace w \|_{L^2(\hK)}^2
+
h \| \left\lA \hlaplace w \right\rA \|_{L^2(\partial \hK)}^2
+
h^{-1} \| \hEn \cdot \lJ\hgrad w \rJ \|_{L^2(\partial \hK)}^2
\end{align}
Note that $\tnorm{\cdot}$ is indeed a norm on $W$ since if $\tnorm{w}=0$ then $w$ must
be the solution to the problem
\begin{alignat}{2}
-\hlaplace w &= 0  \quad &&\text{on $\mathcal{K}_h$} \label{eq:nrmprb1}
\\
\hEn \cdot \lJ \hgrad w \rJ &= 0  \quad &&\text{on $\mathcal{E}_h$}  \label{eq:nrmprb2}
\end{alignat}
Weakly formulating \eqref{eq:nrmprb1} and choosing $w$ as test function gives
\begin{align}
\sum_{K\in\mathcal{K}}(-\hlaplace w,w)_K
&=
\sum_{K\in\mathcal{K}}(\hgrad w,\hgrad w)_K
- (\hEnK \cdot \hgrad w, w)_{\partial K}
\\
&=
\sum_{K\in\mathcal{K}}(\hgrad w,\hgrad w)_K
- \sum_{E\in\mathcal{E}} (\hEn \cdot \lJ \hgrad w \rJ ,w)_{E} = 0
\end{align}
where we use \eqref{eq:tproperty} in the second equality.
By \eqref{eq:nrmprb2} we then have $\|\hgrad w\|_{L^2(\mathcal{K})}^2=0$, and thus $w$ must be a constant function over $\Gamma_h$ and through the extension to $U$ also be constant in $U$. Due to the criteria \eqref{averagesolution} and \eqref{approxcrit}, $w$ must then be the zero function.

Further, we will also need the following energy norm corresponding to \eqref{hnorm} albeit with exact differential operators
and integration over the exact surface
\begin{align} \label{bnorm}
\bnorm{w}^2
:=
\sum_{K \in \mathcal{K}}
\| \blaplace w \|_{L^2(\bK)}^2
+
h \| \left\lA \blaplace w \right\rA \|_{L^2(\partial \bK)}^2
+
h^{-1} \| \bEn \cdot \lJ\bgrad w \rJ \|_{L^2(\partial \bK)}^2
\end{align}
By arguments analogous to the above
$\bnorm{\cdot}$ is also a norm on $W$.
We will later prove that these two norms are actually equivalent
for functions in $W_h$ (see Lemma~\ref{lemma:normequiv}).

As a technical tool in the proof we will also use the following norm
\begin{align}
\| v \|_{H_h^*(\Gamma)}
:=
\bnorm{v}
+
| v |_{H^1(\Gamma)}
+
h | v |_{H_h^2(\Gamma)}
+
h^2 | v |_{H_h^3(\Gamma)}
\label{technicalnorm}
\end{align}



\begin{definition} \label{def:interpolant}
Let $\pi_2$ be the standard continuous piecewise quadratic Lagrange interpolation operator on $\Gamma_h$.
The interpolation operator $\pi:W \rightarrow W_h$ is given by
\begin{align}
\pi w := (\pi_2 w)^\ell
\end{align}
\end{definition}
\begin{remark}
This interpolation can be viewed as defining the nodal values on $\Gamma_h$ by fetching
values on $\Gamma$ by the closest point mapping $\cpm(\bfx)$.
\end{remark}

\begin{definition} \label{def:P0proj}
For any domain $\Sigma$ let $\mathcal{P}_0^\Sigma : L^1(\Sigma) \rightarrow \mathbb{R}$ be the projection onto
the space of constants such that
$
\mathcal{P}_0^\Sigma w = \left| \Sigma \right|^{-1} (w,1)_\Sigma
$
for $w \in L^1(\Sigma)$.
\end{definition}
\begin{remark}
This projection gives the average over the domain and note that it may be used for defining the quotient space
$\| w \|_{L^2(\Sigma)/\mathbb{R}} := \| w - \mathcal{P}_0^\Sigma w \|_{L^2(\Sigma)}$.
\end{remark}


In the next section we establish a number of lemmas needed for the proofs of the main a priori error estimates in Section \ref{section:apriorierrorest}.
\subsection{Prerequisite lemmas} \label{section:preliminaries}

\begin{lemma}[Trace inequalities]
For $v \in W$ the following trace inequalities hold
\begin{align}
\| v \|^2_{L^2(\partial K)} &\lesssim h^{-1} \| v \|^2_{L^2(K)} + h | v |^2_{H^1(K)} 
\label{traceineq}
\\
\| v \|^2_{L^2(\partial \bK)} &\lesssim h^{-1} \| v \|^2_{L^2(\bK)} + h | v |^2_{H^1(\bK)}
\label{btraceineq}
\end{align}
\end{lemma}
\begin{proof}
\proofterm{Inequality (\ref{traceineq}):}
This well known trace inequality follows by affinely mapping $K$ to a reference element $K_{\text{ref}}$, applying the
trace inequality $\|v\|_{L^2(\partial K_{\text{ref}})}^2 \linebreak \lesssim \|v\|_{L^2(K_{\text{ref}})} \|v\|_{H^1(K_{\text{ref}})}  \lesssim \|v\|_{H^1(K_{\text{ref}})}^2$
(see \cite{brennerscott}), and mapping back to $K$.

\proofterm{Inequality (\ref{btraceineq}):}
Due to the extension in $W$, clearly $\| v \|_{L^2(\partial \bK)} \lesssim \| v \|_{L^2(\partial \hK)}$. Applying the original trace inequality \eqref{traceineq} and Lemma~\ref{sobolevestimate}, the trace inequality on curved elements \eqref{btraceineq} immediately follows. 
\end{proof}

Throughout the various parts of the analysis we will make frequent use of the following lemma which gives control over discrete functions $w_h$ in $H^1(\Gamma)$ norm using a duality argument.
\begin{lemma} \label{h1lemma}
For all $w_h \in W_h$ it holds
\begin{align}
\| w_h \|_{L^2(\Gamma) / \mathbb{R}}
+ | w_h |_{H^1(\Gamma)}
 \lesssim \bnorm{w_h}
\end{align}
\end{lemma}

\begin{proof}
We introduce the dual problem
$\blaplace^2 \phi = \psi$ with $(\phi,1)_\Gamma=0$,
where $\psi \in L^2(\Gamma)$ with $(\psi,1)_\Gamma=0$, and for which the stability estimate \eqref{eq:stabest} holds, i.e.
\begin{align}
\| \phi \|_{H^4(\Gamma)} \leq \| \psi \|_{L^2(\Gamma)} \label{stability111}
\end{align}
Integrating by parts twice and then applying Cauchy--Schwarz inequality we get
\begin{align}
\| \psi \|_{H^{1}(\Gamma)}^2
&=
\| \psi \|_{L^2(\Gamma)}^2
+ \sum_{K \in \mathcal{K}}
 \| \bgrad \psi \|_{L^2(\bK)}^2
\\
&=
( \psi , \blaplace^2 \phi )_\Gamma 
+ \sum_{K \in \mathcal{K}}
( \bgrad \psi , \bgrad \psi )_\bK
\\
&=
- ( \bgrad \psi , \bgrad \blaplace \phi )_\Gamma
+ \sum_{K \in \mathcal{K}}
( \bgrad \psi , \bgrad \psi )_\bK
\\
&= \sum_{K \in \mathcal{K}} ( \blaplace \psi , \blaplace \phi )_\bK - \sum_{E \in \mathcal{E}} ( \lJ\bEn \cdot \bgrad \psi\rJ , \blaplace \phi )_\bE 
\\ \nonumber
&\quad - \sum_{K \in \mathcal{K}} ( \psi , \blaplace \psi )_\bK - \sum_{E \in \mathcal{E}} ( \lJ\bEn \cdot \bgrad \psi\rJ , \psi )_\bE
\\
\label{traceref111}
&\lesssim \left( \| \blaplace \psi \|^2_{L^2(\Gamma)} + h^{-1} \| \lJ\bEn \cdot \bgrad \psi\rJ \|^2_{L^2(\mathcal{E}^\ell)} \right)^{1/2}
\\ \nonumber
&\qquad\quad  \cdot
\left(\| \blaplace \phi \|_{H^1(\Gamma)} + \| \psi \|_{H^1(\Gamma)} \right) 
\\
&\lesssim \bnorm{ \psi } \left(\| \blaplace \phi \|_{H^1(\Gamma)} + \| \psi \|_{H^1(\Gamma)} \right)
\end{align}
where we use the trace inequality \eqref{btraceineq}
in \eqref{traceref111}. By \eqref{stability111} we then have
\begin{align}
\left(\| \blaplace \phi \|_{H^1(\Gamma)} + \| \psi \|_{H^1(\Gamma)} \right)
&\lesssim \left(\| \phi \|_{H^3(\Gamma)} + \| \psi \|_{H^1(\Gamma)} \right) \\
&\lesssim \left(\| \psi \|_{L^2(\Gamma)} + \| \psi \|_{H^1(\Gamma)} \right)
\label{seclast111}
\lesssim \| \psi \|_{H^{1}(\Gamma)}
\end{align}
Clearly, for $w_h \in W_h$ we have $w_h|_{\Gamma} \in H^4_h(\Gamma) \cap C^0(\Gamma) \subset L^2(\Gamma)$ and thus we can choose
$\psi = w_h|_\Gamma - \mathcal{P}_0^\Gamma w_h$
where $w_h \in W_h$ which concludes the proof.
\qquad
\end{proof}


\begin{lemma}[Interpolation estimates] \label{lemma:interpolation}
Let $u \in V$ and $\pi: W \rightarrow W_h$ constructed as in Definition \ref{def:interpolant}. The following interpolation estimates then hold
\begin{align}
\tnorm{u - \pi u} &\lesssim h \| u \|_{H^3(\Gamma)} \label{hinterp} \\
\bnorm{u - \pi u} &\lesssim h \| u \|_{H^3(\Gamma)} \label{binterp} \\
\left\|u - \pi u\right\|_{H^*_h(\Gamma)} &\lesssim h \| u \|_{H^3(\Gamma)} \label{xinterp}
\end{align}
for $h < h_0$, with $h_0$ sufficiently small.
\end{lemma}

\begin{proof}
Throughout this proof we will repeatedly use the following standard interpolation estimate
\begin{align}
|u - \pi u|_{H^k(K)} &\lesssim h^{3-k} | u |_{H^3(K)} \label{standardinterp}
 \qquad \text{for $k=1,2,3$}
\end{align}
with mesh independent constants, cf. \cite{claes}.
On curved elements we will also need the corresponding interpolation estimate
\begin{align}
|u - \pi u|_{H^k(\bK)}
\lesssim \sum_{m=1}^k h^{3-m}
| u |_{H^3(\hK)}
\lesssim h^{3-k}
| u |_{H^3(\hK)} \qquad \text{for $k=1,2,3$}
\label{curveinterp}
\end{align}
which directly follows from Lemma~\ref{sobolevestimate} and \eqref{standardinterp}.

\proofterm{Estimate (\ref{hinterp}):}
Establishing estimates
\begin{align}
\sum_{K \in \mathcal{K}} \| \hlaplace (u - \pi u) \|_{L^2(K)}^2
&\lesssim
h^2 | u |_{H^3(\Gamma_h)}^2 \label{helement}
\\
\sum_{K \in \mathcal{K}} h \| \left\lA \hlaplace (u - \pi u) \right\rA \|_{L^2(\partial \hK)}^2
&\lesssim
h^2 | u |_{H^3(\Gamma_h)}^2 \label{haverage}
\\
\sum_{K \in \mathcal{K}} h^{-1} \| \hEn \cdot \lJ\hgrad (u - \pi u)\rJ \|_{L^2(\partial \hK)}^2
&\lesssim
h^2 | u |_{H^3(\Gamma_h)}^2 \label{hjump}
\end{align}
and applying  Lemma~\ref{sobolevestimate} will yield the desired interpolation estimate.
Firstly, estimate \eqref{helement}
directly follows from \eqref{standardinterp} as $\| \hlaplace w \|_{L^2(K)} \leq |w|_{H^2(K)}$.
Secondly, using the triangle inequality on the average in \eqref{haverage} and
on the jump in \eqref{hjump} it suffices to show the elementwise estimates
\begin{align}
h \| \hlaplace (u - \pi u) \|_{L^2(\partial \hK)}^2
&\lesssim h^2 | u |_{H^3(K)}^2
\\
h^{-1} \| \hEn \cdot \hgrad (u - \pi u) \|_{L^2(\partial \hK)}^2
&\lesssim h^2 | u |_{H^3(K)}^2 \label{sectrace}
\end{align}
to prove estimates \eqref{haverage} and \eqref{hjump}.
Using \eqref{traceineq} we get
\begin{align}
h \| \hlaplace (u - \pi u) \|_{\partial \hK}^2
&\lesssim
\| \hlaplace (u - \pi u) \|_{\hK}^2
+ h^2
\| \hgrad \hlaplace (u - \pi u) \|_{\hK}^2
\\
&\lesssim
| u - \pi u |_{H^2(\hK)}^2
+ h^2
| u |_{H^3(\hK)}^2
\lesssim
h^2
| u |_{H^3(\hK)}^2
\end{align}
where we used \eqref{standardinterp} and that $\hlaplace (\pi u)$ is constant.
To show \eqref{sectrace} we again apply the trace inequality \eqref{traceineq}
and interpolation estimate \eqref{standardinterp}
\begin{align}
h^{-1} \| \hEn \cdot \hgrad (u - \pi u) \|_{L^2(\partial \hK)}^2
&\leq
h^{-1} \| \hgrad (u - \pi u) \|_{L^2(\partial \hK)}^2
\\&\lesssim
h^{-2} \| \hgrad (u - \pi u) \|_{L^2(\hK)}^2
+ 
| \hgrad (u - \pi u) |_{H^1(\hK)}^2
\\&\lesssim
h^{-2} | u - \pi u |_{H^1(\hK)}^2
+ 
| u - \pi u |_{H^2(\hK)}^2
\lesssim h^2
| u |_{H^3(K)}^2
\end{align}
Estimates \eqref{helement}--\eqref{hjump} are thereby established
which concludes the proof of estimate \eqref{hinterp}.

\proofterm{Estimate (\ref{binterp}):}
This estimate follows by calculations analogous to those in the proof above for the interpolation estimate on a flat element \eqref{hinterp}, albeit using the trace inequality \eqref{btraceineq} and interpolation estimate \eqref{curveinterp} for curved triangles.

\proofterm{Estimate (\ref{xinterp}):} 
This estimate directly follows from \eqref{binterp} and
\eqref{curveinterp}.
\qquad
\end{proof}

\begin{lemma}[Inverse estimates] \label{lemma:hinvest}
For $v_h \in W_h$ and $k=1,2,\dots$ the following inverse estimates hold
\begin{align}
\label{eq:inverseH1}
h^{k-1} | v_h |_{H^k_h(\Gamma)} &\lesssim | v_h |_{H^1(\Gamma)}
\\
\label{eq:inverseest}
h \sum_{K \in \mathcal{K}} \| \left\lA \hlaplace v_h \right\rA \|_{L^2(\partial K)}^2
&\lesssim
\sum_{K \in \mathcal{K}} \| \hlaplace v_h \|_{L^2(K)}^2
\\
\label{eq:binverseest}
h \sum_{K \in \mathcal{K}} \| \left\lA \blaplace v_h \right\rA \|_{L^2(\partial \bK)}^2
&\lesssim
\sum_{K \in \mathcal{K}}
\left(
\left\| \blaplace v_h \right\|_{L^2(\bK)}^2
+
h | v_h |_{H^1(K^\ell)}^2
\right)
\end{align}
with constants independent of the meshsize $h$ and the parameter $\beta$.
\end{lemma}
\begin{proof}
\proofterm{Estimate (\ref{eq:inverseH1}):}
By Lemma~\ref{sobolevestimate} we have
$
| v_h |_{H^k_h(\Gamma)} \lesssim \sum_{m=1}^k | v_h |_{H^m_h(\Gamma_h)}
$ and thus, as $v_h$ is continuous, the estimate follows from establishing the elementwise estimate
\begin{align}
h^{k-1} | v_h |_{H^k(\hK)} \lesssim | v_h |_{H^1(K)}
\label{elementwiseInvh}
\end{align}
and again applying Lemma~\ref{sobolevestimate}. To show estimate \eqref{elementwiseInvh}
we first note that $v_h|_\hK \in \mathcal{P}_2(\hK)$. Affinely mapping to a reference triangle $\hK_{\text{ref}}$ yields a mapped function $\widehat{v_h} \in \mathcal{P}_2(\hK_{\text{ref}})$. On the reference triangle we now establish the inequality
$| \widehat{v_h} |_{H^k(\hK_{\text{ref}})} \lesssim | \widehat{v_h} |_{H^1(\hK_{\text{ref}})}$ holds.
If the right hand side is zero then $\widehat{v_h}$ must be a constant function and in turn the left hand side must also be zero. Due to the finite dimensionality of $\mathcal{P}_2(\hK_{\text{ref}})$ the inequality on the reference domain follows. Affinely mapping back to $K$ yields \eqref{elementwiseInvh}
and in turn estimate \eqref{eq:inverseH1}.

\proofterm{Estimate (\ref{eq:inverseest}):}
Applying the triangle inequality to each average term and then the trace inequality \eqref{traceineq} this inverse estimate follows as $v_h \in \mathcal{P}_2(K)$.

\proofterm{Estimate (\ref{eq:binverseest}):}
We apply the triangle inequality on the average and note that\\
$\| \blaplace v_h \|_{L^2(\partial \bK)} \lesssim \| (\blaplace v_h)^\ell \|_{L^2(\partial \hK)}$.
In Appendix \ref{appendix:Cx} we in \eqref{eq:lapexpl1} express
$\hlaplace v_h$ for an extended function $v_h$ in terms of the exact operators acting on $v_h$ evaluated
at the exact surface. Reviewing how the terms in this expression scale with $h$, isolating the appropriate term, squaring both sides and applying the elementary inequality
$(a_1+\dots+a_n)^2 \leq n \left( a_1^2 + \dots + a_n^2 \right)$, the following pointwise estimates follow
\begin{align}
\left( (\blaplace v_h)^\ell \right)^2
&\lesssim \left( \hlaplace v_h \right)^2
+ h^2 \left( (D_\Gamma v_h)^\ell \right)^2
+ h^4 \left( (D_\Gamma^2 v_h)^\ell \right)^2
\label{skfjs}
\\
\left( \hlaplace v_h \right)^2
&\lesssim \left( (\blaplace v_h)^\ell \right)^2
+ h^2 \left( (D_\Gamma v_h)^\ell \right)^2
+ h^4 \left( (D_\Gamma^2 v_h)^\ell \right)^2
\label{uisvb}
\end{align}
Integrating \eqref{skfjs} over a curved element boundary $\partial\bK$ and applying the trace inequality \eqref{btraceineq} on non-Laplacian terms yield
\begin{align}
\| \blaplace v_h \|_{L^2(\partial\bK)}^2
&\lesssim
\| \hlaplace v_h \|_{L^2(\partial\hK)}^2
+ h^2 \| D_\Gamma v_h \|_{L^2(\partial\bK)}^2
+ h^4 \| D_\Gamma^2 v_h \|_{L^2(\partial\bK)}^2
\\&\lesssim
\| \hlaplace v_h \|_{L^2(\partial\hK)}^2
+ h | v_h |_{H^1(\bK)}^2
\\&\qquad\nonumber
+ h^3 | v_h |_{H^2(\bK)}^2
+ h^5 |  v_h |_{H^3(\bK)}^2
\\&\lesssim
\| \hlaplace v_h \|_{L^2(\partial\hK)}^2
+ h | v_h |_{H^1(\bK)}^2
\end{align}
where we use \eqref{eq:inverseH1} in the last inequality.
To deal with the first term we establish the elementwise inequality
\begin{align}
\| \hlaplace v_h \|_{L^2(\partial \hK)}^2 \lesssim h^{-1} \| \hlaplace v_h \|_{L^2(\hK)}^2
\end{align}
by mapping to a reference triangle $K_{\text{ref}}$, noting that the inequality on the reference domain holds due to the finite dimensionality of $\mathcal{P}_2(K)$, and mapping back to $K$.
From \eqref{uisvb} we deduce
\begin{align}
h^{-1} \| \hlaplace v_h \|_{L^2(\hK)}^2
&\lesssim
h^{-1} \| \blaplace v_h \|_{L^2(\bK)}^2
+ | v_h |_{H^1(\bK)}^2
\\&\quad\nonumber
+ h^2 | v_h |_{H^2(\bK)}^2
+ h^4 |  v_h |_{H^3(\bK)}^2
\\&\lesssim
h^{-1} \| \blaplace v_h \|_{L^2(\bK)}^2
+ | v_h |_{H^1(\bK)}^2
\end{align}
where we again use \eqref{eq:inverseH1} in the last inequality.
Combining the above results and summing over all elements give
estimate \eqref{eq:binverseest}.
\qquad
\end{proof}

\begin{lemma}\label{lemma:techsmallerthan}
It holds that
\begin{alignat}{2}
\| v_h \|_{H_h^*(\Gamma)} &\lesssim \bnorm{ v_h} \qquad &&\text{for all $v_h \in W_h$}
\label{technormVh}
\\
\bnorm{ v } \leq \| v \|_{H_h^*(\Gamma)} &\lesssim \| v \|_{H^3(\Gamma)} \qquad &&\text{for all $v \in V$}
\label{technormV}
\end{alignat}
where $\| \cdot \|_{H_h^*(\Gamma)}$ and $\bnorm{\cdot}$ are defined in
\eqref{technicalnorm} and \eqref{bnorm}, respectively.
\end{lemma}
\begin{proof}
By \eqref{technicalnorm} we have  $\| v \|_{H_h^*(\Gamma)}
:=
\bnorm{v}
+
| v |_{H^1(\Gamma)}
+
h | v |_{H_h^2(\Gamma)}
+
h^2 | v |_{H_h^3(\Gamma)}$.
For the first estimate we need to limit the last three terms by $\bnorm{v}$.
This is established by applying \eqref{eq:inverseH1} and Lemma~\ref{h1lemma}.

In the second estimate the first inequality follows trivially from the definition of $\| v \|_{H_h^*(\Gamma)}$
and the second inequality from
using the trace inequality \eqref{btraceineq} on the term
$h \| \left\lA \hlaplace v \right\rA \|_{L^2(\partial \hK)}^2$ in \eqref{bnorm}.
\end{proof}

In the following lemma we collect two basic results on continuity and coercivity for the method.
\begin{lemma}[Continuity and coercivity of the method] \label{lemma:basics}
\noindent 1.
There are constants which are independent of h but in general depend on $\beta$, such
that
\begin{align}
a(v,w) &\lesssim \bnorm{v} \, \bnorm{w} \label{bcont} \\
a_h(v,w) &\lesssim \tnorm{v} \, \tnorm{w} \label{hcont}
\end{align}
hold for all $v,w \in W$. \\[-0.5em]

\noindent 2.
For $\beta$ sufficiently large the coercivity estimates
\begin{align}
\bnorm{v_h}^2 &\lesssim a(v_h,v_h)
\label{bcoercivity}
\\
\tnorm{v_h}^2 &\lesssim a_h(v_h,v_h)
\label{coercivity}
\end{align}
hold for all $v_h \in W_h$, with positive constants independent of $h$ and $\beta$.
\end{lemma}

\begin{proof}
\proofterm{Part 1:}
Using the Cauchy--Schwarz inequality on each term in $a(v,w)$ and\linebreak $a_h(v_h,w_h)$, respectively, the inequalities readily follow.

\proofterm{Part 2:}
For estimate \eqref{coercivity} we have
\begin{multline}
a_h(v_h,v_h) = 
\sum_{K \in \mathcal{K}} \| \hlaplace v_h \|_{L^2(\hK)}^2
-
2 \sum_{E \in \mathcal{E}}  \left(\left\lA \hlaplace v_h \right\rA, \lJ\hEn \cdot \hgrad v_h \rJ \right)_\hE 
\\
+ \sum_{E \in \mathcal{E}} \beta \left( h^{-1} \lJ\hEn \cdot \hgrad v_h \rJ, \lJ\hEn \cdot \hgrad v_h \rJ \right)_\hE
\end{multline}

Using the Cauchy--Schwarz inequality followed by the 
standard inequality $2 ab < \epsilon a^2 + \epsilon^{-1} b^2$, for 
any positive $\epsilon$, and finally the inverse inequality 
\eqref{eq:inverseest} we obtain
\begin{align}
 & -2 \sum_{E \in \mathcal{E}} 
   ( \left\lA \hlaplace v_h \right\rA , \lJ\hEn \cdot \hgrad v_h \rJ )_\hE  
  \geq 
	\\&\qquad\qquad\qquad\quad \nonumber
	\sum_{K \in \mathcal{K}}
	-\epsilon C \| \hlaplace v_h \|_{L^2(\hK)}^2
       - \epsilon^{-1} 
     h^{-1} \| \lJ\hEn \cdot \hgrad v_h \rJ \|^2_{L^2(\partial \hK)}
\end{align}
Given $c$, with $0<c<1$, we choose $\epsilon C = (1-c)/3$ 
and take $\beta \geq c + \epsilon^{-1}$ 
we obtain the coercivity estimate 
\begin{align}
c \tnorm{v_h}^2 \leq a_h(v_h,v_h) \qquad \text{for all $v_h \in W_h$}
\end{align}
Coercivity estimate \eqref{bcoercivity} follows by analogous arguments and the inverse inequality \eqref{eq:binverseest} in combination with Lemma~\ref{h1lemma}.
\end{proof}

Next we turn to estimating the difference between the exact and approximate bilinear and linear forms
for discrete functions and introduce the following lemma.
\begin{lemma} \label{lemma:adiff}
The following estimates hold for the approximation errors of the approximate linear functional $l_h(\cdot)$ and approximate bilinear form $a_h(\cdot,\cdot)$
\begin{alignat}{2}
| l(w_h) - l_h(w_h) | &\lesssim h^2 \| f \|_{L^2(\Gamma)} \bnorm{ w_h } \label{ldiff}
\qquad&&\text{for $w_h\in W_h$}
\\
| a(v,w) - a_h(u,w) | &\lesssim h \| v \|_{H^*_h(\Gamma)} \| w \|_{H^*_h(\Gamma)} \label{adiff}
\qquad&&\text{for $v,w\in W$}
\\
| a(\omega,\phi) - a_h(\omega,\phi) | &\lesssim h^2 \| \omega \|_{H^4(\Gamma)} \| \phi \|_{H^4(\Gamma)} \label{acontdiff}
\qquad&&\text{for $\omega,\phi\in V$}
\end{alignat}
for $h < h_0$, with $h_0$ sufficiently small.
\end{lemma}

\begin{remark}
While the estimate for the consistency error in bilinear forms \eqref{adiff} is sufficient for proving the correct order
of convergence in energy norm, it is insufficient for proving the correct convergence in $L^2$ norm.
For this reason we also include estimate \eqref{acontdiff} for smooth functions, which we utilize in the proof
of the $L^2$ estimate as we therein consider approximations to functions in $V$.
In the proof of \eqref{acontdiff} it no longer suffices to only compare terms elementwise,
but we also need to consider complete integral expressions where we utilize the high regularity of functions in $V$ and Lemma~\ref{lemma:geomnonstandard}.
\end{remark}

\begin{proof}

\quad
We here prove each estimate but give some of the details in Appendix~\ref{appendix:Cx}.

\proofterm{Estimate (\ref{ldiff}):}
As $0=(1,f)_\Gamma = (\mu_h,f^\ell)_{\Gamma_h}$ we by the Cauchy--Schwarz inequality have
\begin{align}
(1,f^\ell)_{\Gamma_h} = (1 - \mu_h,f^\ell)_{\Gamma_h} \leq |\Gamma_h|^{1/2} \| 1 - \mu_h \|_{L^\infty(\Gamma_h)} \| f^\ell \|_{L^2(\Gamma_h)}
\end{align}
By definition $f_h:=f^\ell - \mathcal{P}_0^{\Gamma_h} f^\ell$. Using $(f,1)_\Gamma=(f_h,1)_{\Gamma_h}=0$, estimate \eqref{muest}, and the Cauchy--Schwarz inequality we get
\begin{align}
&\left| l( w_h ) - l_h( w_h ) \right|
= \left| ( f , w_h )_\Gamma - ( f_h , w_h )_{\Gamma_h} \right| \\
&\qquad\qquad= \left| ( f , w_h - \mathcal{P}_0^\Gamma w_h )_\Gamma - ( f_h , w_h - \mathcal{P}_0^\Gamma w_h )_{\Gamma_h} \right| \\
&\qquad\qquad\leq \left| \left( (\mu_h-1) f^\ell , w_h - \mathcal{P}_0^\Gamma w_h \right)_{\Gamma_h} \right|
+ \left| \left( \mathcal{P}_0^{\Gamma_h} f^\ell , w_h - \mathcal{P}_0^\Gamma w_h \right)_{\Gamma_h} \right|
\\
&\qquad\qquad\lesssim \| 1 - \mu_h \|_{L^\infty(\Gamma_h)} \|f^\ell \|_{L^2(\Gamma_h)} \|w_h - \mathcal{P}_0^\Gamma w_h \|_{L^2(\Gamma_h)} \\
&\qquad\qquad\lesssim h^2 \|f^\ell\|_{L^2(\Gamma_h)} \|w_h - \mathcal{P}_0^\Gamma w_h^\ell \|_{L^2(\Gamma_h)} \\
&\qquad\qquad\lesssim h^2 \|f\|_{L^2(\Gamma)} \|w_h^\ell \|_{L^2(\Gamma)/\mathbb{R}} \label{eq:lestex444}\\
&\qquad\qquad\lesssim h^2 \|f\|_{L^2(\Gamma)} \bnorm{ w_h^\ell }
\end{align}
where we use Lemma \ref{sobolevestimate} in \eqref{eq:lestex444} and Lemma \ref{h1lemma} in the last inequality.

\proofterm{Estimate (\ref{adiff}):}
All terms in $a(v,w) - a_h(v,w)$ can be paired and rewritten as
\begin{align} \label{operatordiff}
\left( \mathcal{A} v , \mathcal{B} w \right)_{\Omega_\Gamma} 
- (\mathcal{A}_h v, \mathcal{B}_h w )_{\Omega_h}
\end{align}
where $\Omega_\Gamma, \Omega_h$ are either $\Gamma,\Gamma_h$ or $\mathcal{E}_\Gamma,\mathcal{E}_h$
depending on the term, $\mathcal{A},\mathcal{B}$ are the differential operators in the term, and $\mathcal{A}_h,\mathcal{B}_h$ are the corresponding approximate operators.
For the first term in \eqref{operatordiff} we change the integration
domain to the approximate surface. Adding and subtracting terms then yields
\begin{align}
\left( \mathcal{A} v , \mathcal{B} w \right)_{\Omega_\Gamma}
&= \left( \mu_h (\mathcal{A} v )^\ell , (\mathcal{B} w )^\ell \right)_{\Omega_h}
\\&= \left( (\mu_h - 1) (\mathcal{A} v )^\ell , (\mathcal{B} w )^\ell \right)_{\Omega_h}
+ \left( (\mathcal{A} v )^\ell , (\mathcal{B} w )^\ell \right)_{\Omega_h}
\end{align}
By further adding and subtracting terms, may express \eqref{operatordiff} as
\begin{align}
\left( \mathcal{A} v , \mathcal{B} w \right)_{\Omega_\Gamma}
- (\mathcal{A}_h v, \mathcal{B}_h w )_{\Omega_h}
&=
\left( (\mu_h - 1) (\mathcal{A} v )^\ell , (\mathcal{B} w )^\ell \right)_{\Omega_h}  \\ \nonumber
&\quad +
\left( (\mathcal{A} v )^\ell - \mathcal{A}_h v , (\mathcal{B} w )^\ell \right)_{\Omega_h}  \\ \nonumber
&\quad +
\left( (\mathcal{A} v )^\ell , (\mathcal{B} w )^\ell - \mathcal{B}_h w \right)_{\Omega_h}  \\ \nonumber
&\quad -
\left( (\mathcal{A} v )^\ell - \mathcal{A}_h v , (\mathcal{B} w )^\ell  - \mathcal{B}_h w \right)_{\Omega_h}  \\
&=I + II + III + IV
\end{align}
 Clearly, due to the bound \eqref{muest} on $(1 - \mu_h)$ and that all terms with $(\mathcal{A} w )^\ell$ and
$(\mathcal{B} w )^\ell$ are included in $\bnorm{w}$, we by the Cauchy--Schwarz inequality have
\begin{align}
\sum_{K\in\mathcal{K}} \left| I \right| \lesssim h^2 \bnorm{v} \bnorm{w}
\leq h^2 \| v \|_{H_h^*(\Gamma)} \| w \|_{H_h^*(\Gamma)}
\label{eq:gfdn98}
\end{align}

For the remaining three terms; $II$, $III$ and $IV$, it suffices to prove that for $v \in W$ we have
\begin{align}
\| (\mathcal{A} v )^\ell - \mathcal{A}_h v \|_{L^2(\Omega_h)} \lesssim h \| v \|_{H_h^*(\Gamma)}
\end{align}
for all operator and domain pairs $\{ \mathcal{A}_h,\Omega_h \}$ present in $a_h(\cdot,\cdot)$.
We collect and prove this inequality for all operator and domain pairs in Lemma~\ref{lemma:elmopdiff}.
By the Cauchy--Schwarz inequality it then follows that
\begin{alignat}{2}
\sum_{K\in\mathcal{K}} \left| II \right| +  \left| III \right| 
&\lesssim h \| v \|_{H_h^*(\Gamma)} \| w \|_{H_h^*(\Gamma)}
\qquad&&\text{for all $v,w \in W$}
\label{eq:908ebn}
\\
\sum_{K\in\mathcal{K}} \left| IV \right| 
&\lesssim h^2 \| v \|_{H_h^*(\Gamma)} \| w \|_{H_h^*(\Gamma)}
\qquad&&\text{for all $v,w \in W$}
\label{eq:49t8u3}
\end{alignat}
From \eqref{eq:gfdn98}, \eqref{eq:908ebn} and \eqref{eq:49t8u3} we now conclude that estimate \eqref{adiff} holds.

\proofterm{Estimate (\ref{acontdiff}):}
As our functions $\eta,\phi\in V\subset W$ we use \eqref{eq:gfdn98} and \eqref{eq:49t8u3} from the proof of the previous estimate in combination with Lemma~\ref{lemma:techsmallerthan} to prove
\begin{align}
\sum_{K\in\mathcal{K}} \left| I \right| + \left| IV \right| 
&\lesssim h^2 \| \eta \|_{H^3(\Gamma)} \| \phi \|_{H^3(\Gamma)}
\qquad\text{for all $\eta,\phi \in V$}
\end{align}
It remains to prove bounds of the same order for terms $II$ and $III$.
Due to the Sobolev embedding $H^3(\Gamma) \hookrightarrow C^1(\Gamma)$ \cite[Thm. 2.20]{aubin1982}
we have $\bEn \cdot \lJ \bgrad \phi \rJ)^\ell = 0$, and the only remaining estimates are
\begin{align}
\left( ( \blaplace \eta )^\ell ,
( \blaplace \phi )^\ell
- \hlaplace \phi^\ell \right)_{\Gamma_{h}}
&\lesssim h^2 \| \eta \|_{H^3(\Gamma)} \| \phi \|_{H^3(\Gamma)}
\\
\left( (\left\lA \blaplace \eta \right\rA)^\ell ,
(\bEn \cdot \lJ \bgrad \phi \rJ)^\ell
- \hEn \cdot \lJ \hgrad \phi^\ell \rJ \right)_{\mathcal{E}_{h}}
&\lesssim h^2 \| \eta \|_{H^4(\Gamma)} \| \phi \|_{H^4(\Gamma)}
\end{align}
which are proven in Lemma~\ref{lemma:specintegrals}.
This completes the proof.
\qquad
\end{proof}

A consequence of the above proof is that for functions in $W_h$ the norm $\bnorm{ \cdot }$, using the exact differential operators \eqref{bnorm},
is equivalent to the norm $\tnorm{ \cdot }$, using approximate operators \eqref{hnorm}. We formulate this in the next lemma.
\begin{lemma}(Equivalence of norms) \label{lemma:normequiv}
For $w_h \in W_h$ the norms $\bnorm{ w_h }$ and $\tnorm{ w_h }$ are equivalent, i.e.
\begin{align}
\tnorm{ w_h } \lesssim \bnorm{ w_h } \lesssim \tnorm{ w_h }
\end{align}
for $h<h_0$, with $h_0$ sufficiently small.
\end{lemma}

\begin{proof}
The left inequality follows from coercivity \eqref{coercivity}, Lemma~\ref{lemma:adiff}, Lemma~\ref{lemma:techsmallerthan}, and continuity \eqref{bcont}
\begin{align}
\tnorm{ w_h }^2 &\lesssim a_h(w_h,w_h)
\\&= a_h(w_h,w_h) - a(w_h,w_h) + a(w_h,w_h)
\lesssim
h \bnorm{w_h}^2 + \bnorm{w_h}^2
\end{align}
for $h$ smaller than some $h_0$.
By coercivity \eqref{bcoercivity}, Lemma~\ref{lemma:adiff}, Lemma~\ref{lemma:techsmallerthan}, and continuity \eqref{hcont} we have
\begin{align}
\bnorm{ w_h }^2 &\lesssim a(w_h,w_h)
\\&= a(w_h,w_h) - a_h(w_h,w_h) + a_h(w_h,w_h)
\lesssim
h \bnorm{w_h}^2 + \tnorm{w_h}^2
\end{align}
and by a simple kick-back argument the right inequality follows for $h$ smaller than some $h_0$.
\end{proof}



\subsection{Main theorems} \label{section:apriorierrorest}

The foundation of the main proof is the first Strang lemma as given in \cite{Ciarlet2002}.
\begin{lemma}[First Strang lemma] \label{lemma:firststrang}
Consider a family of discrete problems for which the associated approximate bilinear forms are uniformly $W_h$-elliptic.
Then there exists a constant independent of the space $W_h$ such that
\begin{align}
\bnorm{ u - u_h } \lesssim 
\Bigg(&
\inf_{v_h \in W_h}
\left(
\bnorm{ u - v_h }
+
\sup_{w_h \in W_h} \frac{| a(v_h,w_h) - a_h(v_h,w_h) |}{\bnorm{ w_h }}
\right)
\\ \nonumber
&+ \sup_{w_h \in W_h} \frac{| l(w_h) - l_h(w_h) |}{\bnorm{ w_h }} \Bigg)
\end{align}
\end{lemma}

We now turn to presenting our main a priori error estimate.
\begin{theorem}[Error estimate in energy norm] \label{maintheorem}
Let $u$ be the exact solution  to $a(u,v)=l(v)$ and let $u_h$ be the finite element solution to the
approximate problem $a_h(u_h,v_h)=l_h(v_h)$ where $\beta$ in $a_h(\cdot,\cdot)$ is chosen sufficiently large for coercivity to hold (see Lemma \ref{lemma:basics}). For $h < h_0$, with $h_0$ small enough, the following error estimate holds
\begin{align}
\bnorm{ u - u_h } \lesssim h \| f \|_{L^2(\Gamma)}
\end{align}
\end{theorem}

\begin{proof}
By coercivity \eqref{bcoercivity}, continuity \eqref{bcont}, weak formulation \eqref{extendedweak}, and method formulation \eqref{approxweak}
the first Strang lemma above holds in our setting. Choosing $v_h = \pi u$ to handle the infimum yields an inequality with three independent terms
\begin{align}
\bnorm{ u - u_h } \lesssim I + II + III
\end{align}
where we will show that there exists constants independent of the meshsize $h$ such that the following estimates hold
\begin{align}
I &:=
\bnorm{ u - \pi u }
\lesssim h \| u \|_{H^3(\Gamma)}
\\
II &:=
\sup_{w_h \in W_h} \frac{| a(\pi u,w_h) - a_h(\pi u,w_h) |}{\bnorm{ w_h }}
\lesssim h \| u \|_{H^3(\Gamma)}
\\
III &:=
\sup_{w_h \in W_h} \frac{| l(w_h) - l_h(w_h) |}{\bnorm{ w_h }}
\lesssim h^2 \| f \|_{L^2(\Gamma)}
\end{align}
and recall that $\| u \|_{H^4(\Gamma)} \lesssim \| f \|_{L^2(\Gamma)}$ by the stability estimate \eqref{eq:stabest}.

\proofterm{Term I:}
This estimate directly follows from Lemma \ref{lemma:interpolation} (interpolation).

\proofterm{Term II:}
By Lemma~\ref{lemma:adiff} and Lemma~\ref{lemma:techsmallerthan} we have
$| a(\pi u,w_h) - a_h(\pi u,w_h) | \linebreak \lesssim h \bnorm{ \pi u } \, \bnorm{ w_h }$
yielding
\begin{align}
II \lesssim \frac{h \bnorm{ \pi u } \, \bnorm{ w_h }}{\bnorm{ w_h }}
\lesssim h \bnorm{ \pi u }
\leq h \left( \bnorm{ u - \pi u } + \bnorm{ u } \right) \lesssim h \| u \|_{H^3(\Gamma)}
\end{align}
where we use the triangle inequality, interpolation \eqref{binterp}, and  Lemma~\ref{lemma:techsmallerthan}.

\proofterm{Term III:}
This estimate directly follows from Lemma \ref{lemma:adiff}.
\end{proof}

Next we prove an a priori estimate in $L^2$ norm using a duality argument (Aubin--Nitsche's trick).
We assume that for all $\psi \in W$ with $(\psi,1)_\Gamma=0$ there is a $\phi \in V$ with $(\phi,1)_\Gamma=0$ such
that
\begin{align} \label{dualprob2}
a(v,\phi) = (v,\psi)_\Gamma  \qquad \text{for all $v\in W$}
\end{align}
for which the stability estimate \eqref{eq:stabest} holds, i.e.
\begin{align} \label{stabest2}
\| \phi \|_{H^4(\Gamma)} \lesssim \| \psi \|_{L^2(\Gamma)}
\end{align}
\begin{theorem}[Error estimate in $L^2$ norm] \label{lemma:L2Gamma}
Given the assumptions of Theorem \ref{maintheorem} and that the stability estimate \eqref{stabest2} holds
we for $h<h_0$, with $h_0$ small enough, have the following error estimate
\begin{align}
\| u - u_h \|_{L^2(\Gamma) / \mathbb{R}} &\lesssim h^2 \| f \|_{L^2(\Gamma)} \label{est:L2Gamma}
\end{align}
\end{theorem}
\begin{proof}
Let $v= \psi=u-u_h - \mathcal{P}_0^{\Gamma} (u-u_h)$. By \eqref{dualprob2} we then have
\begin{align}
\| u - u_h \|_{L^2(\Gamma) / \mathbb{R}}^2 
&=
a(u - u_h,\phi)
=
a(u - u_h,\phi - \pi\phi) + a(u - u_h,\pi\phi)
\label{eq:L2firstline}
\end{align}
where the first term
by continuity \eqref{bcont}, Theorem~\ref{maintheorem}, and interpolation \eqref{binterp}
is limited by
\begin{align}
a(u - u_h,\phi - \pi\phi)
&\leq
\bnorm{ u - u_h } \bnorm{ \phi - \pi\phi }
\lesssim
h^2 \| f \|_{L^2(\Gamma)} \| \phi \|_{H^3(\Gamma)}
\end{align}
For the second term in \eqref{eq:L2firstline} we by \eqref{extendedweak}, \eqref{approxweak}, and adding and subtracting terms have
\begin{align}
a(u - u_h,\pi\phi) &= l(\pi\phi) - a(u_h,\pi\phi)
\\&=
\left( l(\pi\phi)-l_h(\pi\phi) \right)
+  \left( a_h(u_h,\pi\phi) - a(u_h,\pi\phi) \right)
\label{eq:gkhsgh}
\end{align}
where we note that
\begin{align}
l(\pi\phi)-l_h(\pi\phi)
&\lesssim h^2 \| f \|_{L^2(\Gamma)} \bnorm{ \pi \phi }
\\&\leq
h^2 \| f \|_{L^2(\Gamma)} \left(\bnorm{ \phi } + \bnorm{ \phi - \pi\phi }\right)
\\&\lesssim
h^2 \| f \|_{L^2(\Gamma)} \| \phi \|_{H^3(\Gamma)}
\end{align}
through Lemma~\ref{lemma:adiff}, the triangle inequality and interpolation \eqref{binterp}.
Directly applying \eqref{adiff} of Lemma~\ref{lemma:adiff} also to the second term of \eqref{eq:gkhsgh}
only yields $\mathcal{O}(h)$ which is insufficient to prove this theorem.
Instead we shall utilize that $u_h$ and $\pi\phi$
are approximations to functions $u$ and $\phi$, both encompassing high regularity.
By adding and subtracting terms we rewrite the second term of \eqref{eq:gkhsgh} as
the sum of three terms
\begin{align}
a_h(u_h,\pi\phi) - a(u_h,\pi\phi)
&=
\left( a_h(u_h,\pi\phi - \phi) - a(u_h,\pi\phi - \phi) \right) \label{eq:termX1}
\\&\quad+
\left( a_h(u_h-u,\phi) - a(u_h-u,\phi) \right) \nonumber 
\\&\quad+
\left( a_h(u,\phi) - a(u,\phi) \right) \nonumber 
\\&= I + II + III
\end{align}
\proofterm{Term I:} By Lemma~\ref{lemma:adiff} and the interpolation estimate \eqref{xinterp} we have
\begin{align}
a_h(u_h,\pi\phi - \phi) - a(u_h,\pi\phi - \phi)
&\lesssim h
\left\| u_h \right\|_{H_h^*(\Gamma)}
\left\| \phi - \pi\phi \right\|_{H_h^*(\Gamma)}
\\&\lesssim h^2
\| u \|_{H^3(\Gamma)}
\| \phi \|_{H^3(\Gamma)}
\end{align}
where we in the last inequality use
\begin{align}
\| u_h \|_{H^*_h(\Gamma)}
\leq
\| u \|_{H^*_h(\Gamma)} + \| u - u_h \|_{H^*_h(\Gamma)}
\lesssim
\| u \|_{H^3(\Gamma)}
\end{align}
which follows from Lemma~\ref{lemma:techsmallerthan} and interpolation \eqref{xinterp}.

\proofterm{Term II:}
By Lemma~\ref{lemma:adiff} we have
\begin{align}
a_h(u_h - u,\phi) - a(u_h - u,\phi)
&\lesssim h
\left\| u - u_h \right\|_{H_h^*(\Gamma)}
\left\| \phi \right\|_{H_h^*(\Gamma)}
\\&\lesssim h^2
\| u \|_{H^3(\Gamma)}
\| \phi \|_{H^3(\Gamma)}
\end{align}
where $\left\| u - u_h \right\|_{H_h^*(\Gamma)}$ is limited using the triangle inequality, Lemma~\ref{lemma:techsmallerthan},  interpolation estimates, and Theorem~\ref{maintheorem} such that
\begin{align}
\left\| u - u_h \right\|_{H_h^*(\Gamma)}
&\leq
\left\| u - \pi u \right\|_{H_h^*(\Gamma)}
+
\left\| \pi u - u_h \right\|_{H_h^*(\Gamma)}
\\&\lesssim
h \| u \|_{H^3(\Gamma)}
+
\bnorm{\pi u - u_h}
\\&\leq
h \| u \|_{H^3(\Gamma)}
+
\bnorm{u - \pi u}
+
\bnorm{u - u_h}
\lesssim h \| u \|_{H^3(\Gamma)}
\end{align}

\proofterm{Term III:} The estimate
\begin{align}
a_h(u,\phi) - a(u,\phi)
\lesssim
h^2
\left\| u \right\|_{H^4(\Gamma)}
\left\| \phi \right\|_{H^4(\Gamma)}
\end{align}
follows directly from Lemma~\ref{lemma:adiff}. Collecting the above results and using
the stability estimate \eqref{eq:stabest}, i.e. $\| u \|_{H^4(\Gamma)} \lesssim \| f \|_{L^2(\Gamma)}$, yields the estimate
\begin{align}
\| u - u_h \|_{L^{2}(\Gamma)/\mathbb{R}}^2
&\lesssim h^2
\| f \|_{L^2(\Gamma)}
\left\| \phi \right\|_{H^4(\Gamma)}
\lesssim h^2
\| f \|_{L^2(\Gamma)}
\| \psi \|_{L^{2}(\Gamma)}
\end{align}
where we also use the stability estimate for the dual solution \eqref{stabest2}. Recalling that $\psi=u-u_h - \mathcal{P}_0^{\Gamma} (u-u_h)$
concludes the proof of the theorem.
\end{proof}

We now introduce the corresponding $L^2$ estimate on $\Gamma_h$ which is more practical as the
exact surface may be unknown or not easily integrated.
\begin{cor} \label{lemma:L2Gammah}
Given the assumptions of Theorem \ref{maintheorem} and that the stability estimate \eqref{stabest2} holds
we for $h < h_0$, with $h_0$ small enough, have the error estimate
\begin{align}
\| u - u_h \|_{L^2(\Gamma_h) / \mathbb{R}} &\lesssim h^2 \| f \|_{L^2(\Gamma)} \label{est:L2Gammah}
\end{align}
\end{cor}
\begin{proof}
By the triangle inequality we have
\begin{align}
\| u - u_h \|_{L^2(\Gamma_h) / \mathbb{R}}
&\leq
\| u - u_h - \mathcal{P}_0^\Gamma (u - u_h) \|_{L^2(\Gamma_h)}
\\&\quad+ \nonumber
\| (\mathcal{P}_0^\Gamma - \mathcal{P}_0^{\Gamma_h}) (u - u_h) \|_{L^2(\Gamma_h)}
\end{align}
where the first term after application of Lemma \ref{sobolevestimate} is limited
through Theorem~\ref{lemma:L2Gamma}.
For any function $w \in W$ we by the triangle inequality have
\begin{align}
\| (\mathcal{P}_0^\Gamma - \mathcal{P}_0^{\Gamma_h}) w \|_{L^2(\Gamma_h)}
&=
\left\|
|\Gamma |^{-1} (w,1)_\Gamma - |\Gamma_h |^{-1} (w,1)_{\Gamma_h}
 \right\|_{L^2(\Gamma_h)}
\\&\leq
\left\|
(|\Gamma |^{-1} - |\Gamma_h |^{-1}) (w,1)_\Gamma  \right\|_{L^2(\Gamma_h)}
\\&\quad \nonumber
+
\left\| |\Gamma_h |^{-1} ((1-\mu_h)w,1)_{\Gamma_h}
 \right\|_{L^2(\Gamma_h)}
\\&\lesssim h^2 \| w \|_{L^2(\Gamma_h)}
\end{align}
where we use $1-\frac{|\Gamma|}{|\Gamma_h|}=(1-\mu_h,1)_{\Gamma_h}$, $|\Gamma_h|\leq|\Gamma|$, and
the bound \eqref{muest} for $1-\mu_h$
in the last inequality.
Choosing $w=u-u_h$ gives
by the triangle inequality, $\mathcal{P}_0^{\Gamma_h} u_h = 0$, and Lemma \ref{sobolevestimate}
that
\begin{align}
\| u - u_h \|_{L^2(\Gamma_h)}
&\leq
\| u - u_h - \mathcal{P}_0^{\Gamma_h}(u - u_h) \|_{L^2(\Gamma_h)} + \| \mathcal{P}_0^{\Gamma_h} u \|_{L^2(\Gamma_h)}
\\&\leq
\| u - u_h \|_{L^2(\Gamma_h) / \mathbb{R}}
+
c \| u \|_{L^2(\Gamma)}
\end{align}
which after a simple kick-back argument and the stability estimate \eqref{eq:stabest} concludes the proof.
\qquad
\end{proof}

We now turn to our numerical experiments where we present convergence studies in $L^2(\Gamma_h)$ norm to confirm the estimate given by Corollary \ref{lemma:L2Gammah}.

\section{Numerical results} \label{section:numerics}


\subsection{Model problems}

For the numerical results we consider two problems with the same geometries and solutions as the model problems considered in \cite{Olshanskii2009} for the Laplace--Beltrami-problem. The two geometries and solutions are illustrated in Figure~\ref{fig:modelsol}. We analytically calculate the appropriate load functions for the biharmonic problem by inserting the prescribed solutions into the equation.

In the first model problem we consider a sphere with radius $r=1$.
The spherical coordinates $\{\theta,\phi\}$ for the sphere surface are defined such that the corresponding Cartesian coordinates are expressed
\begin{align}
\left\{ 
x = r \sin(\theta) \cos(\phi) \, , \
y = r \sin(\theta) \sin(\phi) \, , \
z = r \cos(\theta)
\right\}
\end{align}
with $0 \leq \theta < \pi$ and $0 \leq \phi \leq 2\pi$.
Given
$f = - 12 r^{-2} \sin(\phi) \sin(\theta)^3 (4 \sin(\phi)^2 - 3)$
we then have the analytical solution
$u = r^{-3} (3 x^2 y - y^3)$.
In the second model problem we consider a torus with $R=1$, $r=0.6$, with
toroidal coordinates $\{\theta,\phi\}$ for the torus surface defined such that the corresponding Cartesian coordinates are given by
\begin{align}
\left\{
x = (R+r \cos(\theta)) \cos(\phi) \, , \
y = (R+r \cos(\theta)) \sin(\phi) \, , \
z = r \sin(\theta)
\right\}
\end{align}
with $0 \leq \theta < 2\pi$ and $0 \leq \phi < 2\pi$.
Using $f$ defined through the {\sc Matlab}-code in Appendix \ref{app:matlab} we have
the analytical solution
$u = \sin(3 \phi) \cos(3\theta + \phi)$.

\begin{figure}
\centering
\subfigure[Sphere]{
\includegraphics[width=0.40\linewidth]{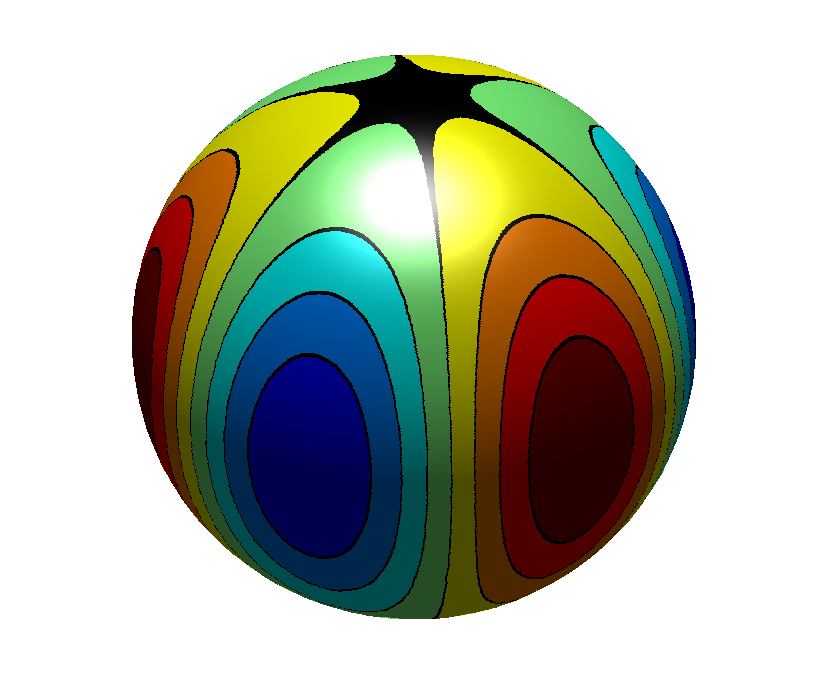}
\label{fig:subfig1}
}
\subfigure[Torus]{
\includegraphics[width=0.45\linewidth]{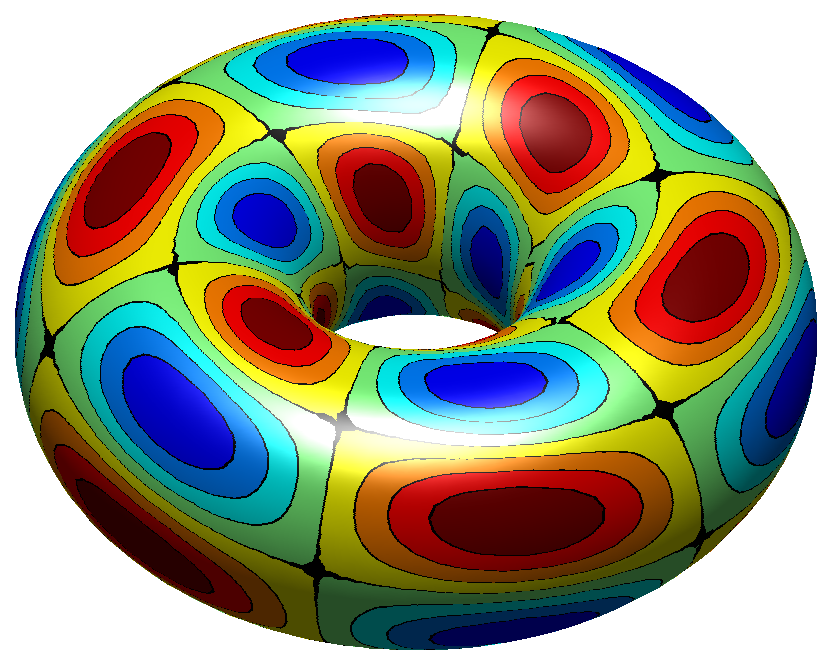}
\label{fig:subfig2}
}
\caption{Illustration of model problem solutions.}
\label{fig:modelsol}
\end{figure}


\subsection{Convergence}

For the convergence study on both model problems we used structured meshes illustrated in Figure~\ref{fig:examplemesh} and unstructured meshes illustrated in Figure~\ref{fig:unexamplemesh}. The unstructured meshed were created by random
perturbation of the vertices in the structured meshes.
Using a penalty parameter $\beta=10$, a choice which we motivate in the next section, we in Figure \ref{fig:conv} present the results from our convergence study in $L^2(\Gamma_h)$ norm for the two model problems on the structured meshes and in Figure~\ref{fig:conv2} the corresponing results on the unstructured meshes. The number of degrees of freedom in this study range from 0.8k to 190k in the sphere model problem and from 1.6k to 340k in the torus model problem, and the number of elements are approximately half of those numbers.

The results in figures \ref{fig:conv} and \ref{fig:conv2} indicate that the order of convergence is 2 in $L^2(\Gamma_h)$ norm which gives confirmation to the sharpness of the $L^2(\Gamma_h)$ estimate presented in Corollary~\ref{lemma:L2Gammah}.
We note more fluctuations in the torus model problem which we assume are due to the more complex
geometry to approximate and also a more complicated load function and analytical solution as suggested by the illustrations of the solutions in Figure~\ref{fig:modelsol}.

\begin{figure}
\centering
\subfigure[Sphere]{
\includegraphics[width=0.32\linewidth]{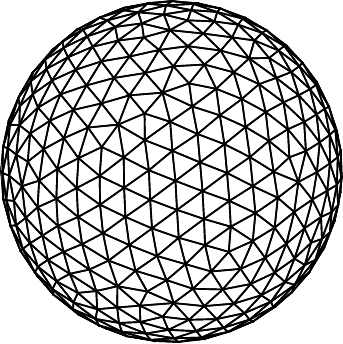}
\label{fig:subbfig1}
}
\quad
\subfigure[Torus]{
\includegraphics[width=0.50\linewidth]{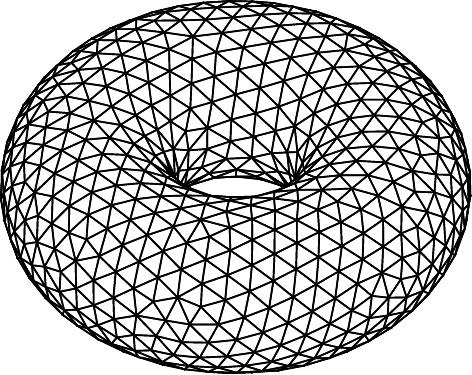}
\label{fig:subbfig2}
}
\caption{Strucutred meshes ($h=0.2$).}
\label{fig:examplemesh}
\end{figure}

\begin{figure}
\centering
\subfigure[Sphere]{
\includegraphics[width=0.32\linewidth]{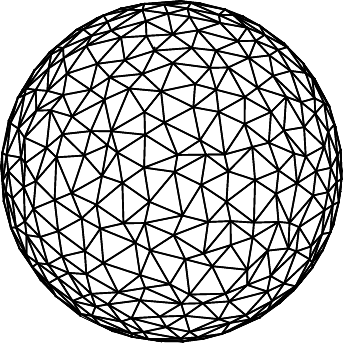}
\label{fig:subbfig3}
}
\quad
\subfigure[Torus]{
\includegraphics[width=0.50\linewidth]{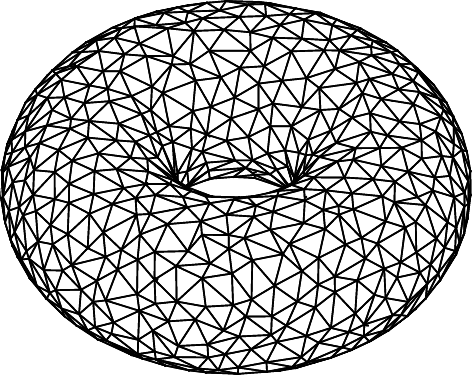}
\label{fig:subbfig4}
}
\caption{Unstructured meshes constructed by randomly moving triangle vertices ($h=0.2$).}
\label{fig:unexamplemesh}
\end{figure}

\begin{figure}
\centering
\subfigure[Sphere]{
   \includegraphics[width=0.45\linewidth]{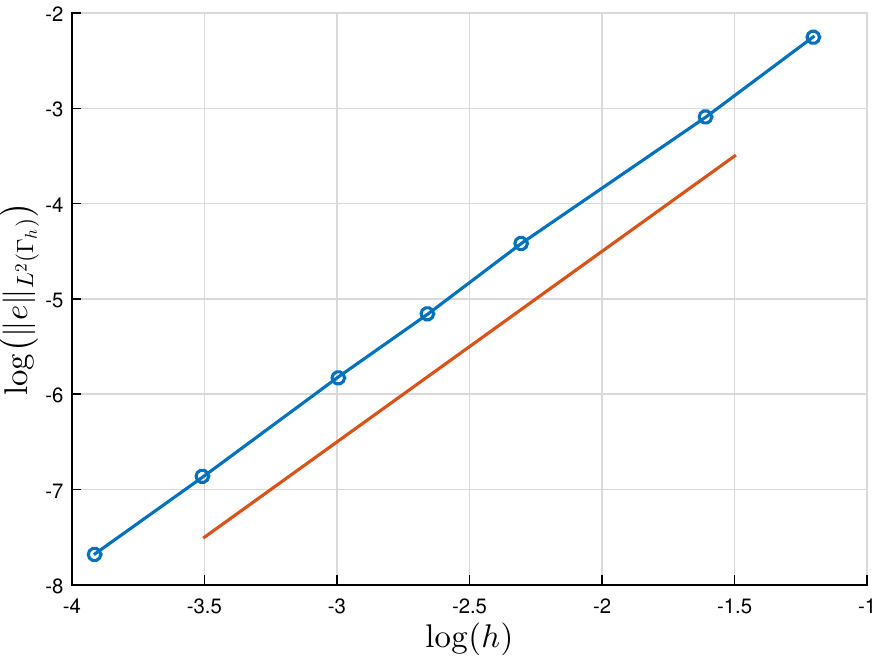}
   \label{fig:sphereconv}
 }
 \subfigure[Torus]{
   \includegraphics[width=0.45\linewidth]{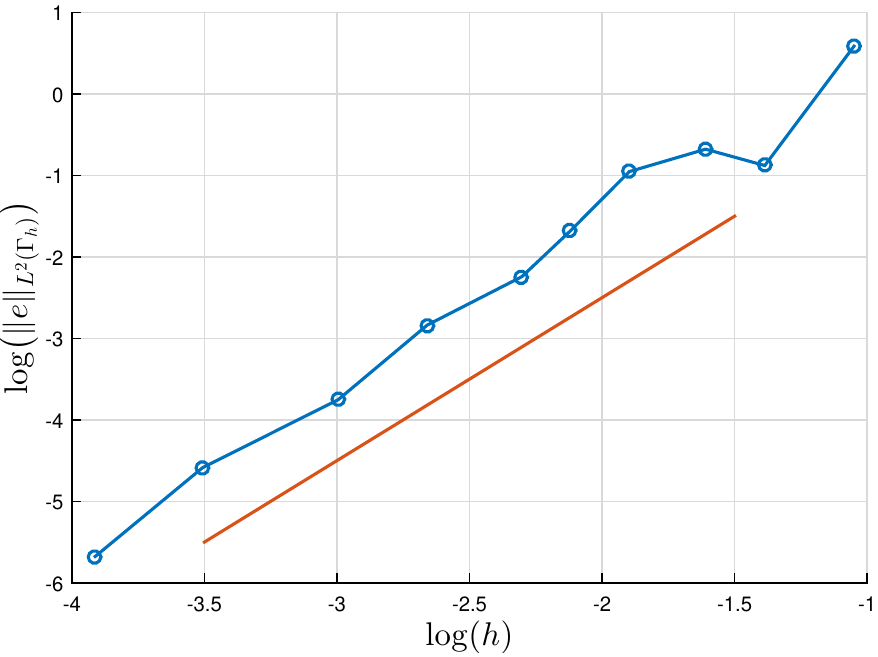}
   \label{fig:torusconv}
 }
\caption{Convergence studies in $L^2(\Gamma_h)$ norm for the two model problems using $\beta=10$. For comparison we include red reference lines with slope 2.}
\label{fig:conv}
\end{figure}

\begin{figure}
\centering
\subfigure[Sphere]{
   \includegraphics[width=0.45\linewidth]{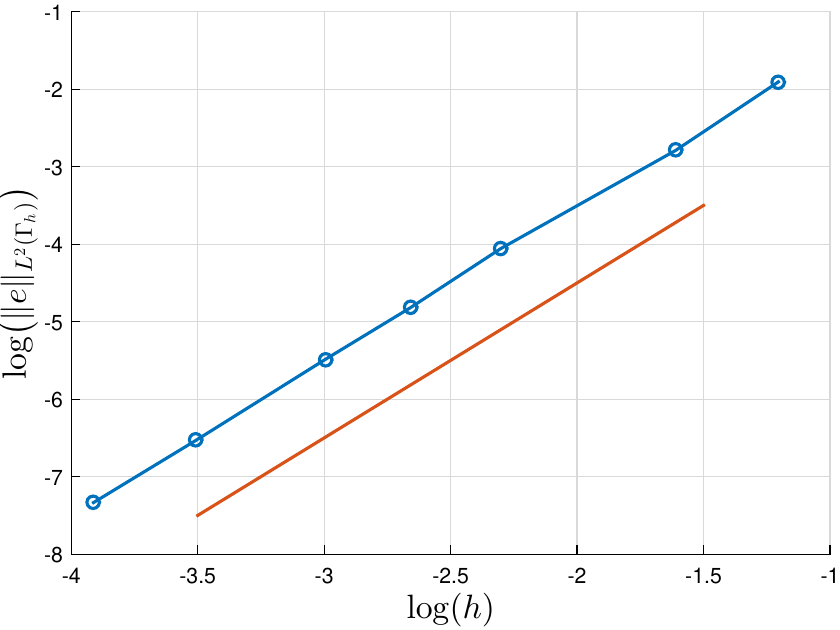}
   \label{fig:sphereconv2}
 }
 \subfigure[Torus]{
   \includegraphics[width=0.45\linewidth]{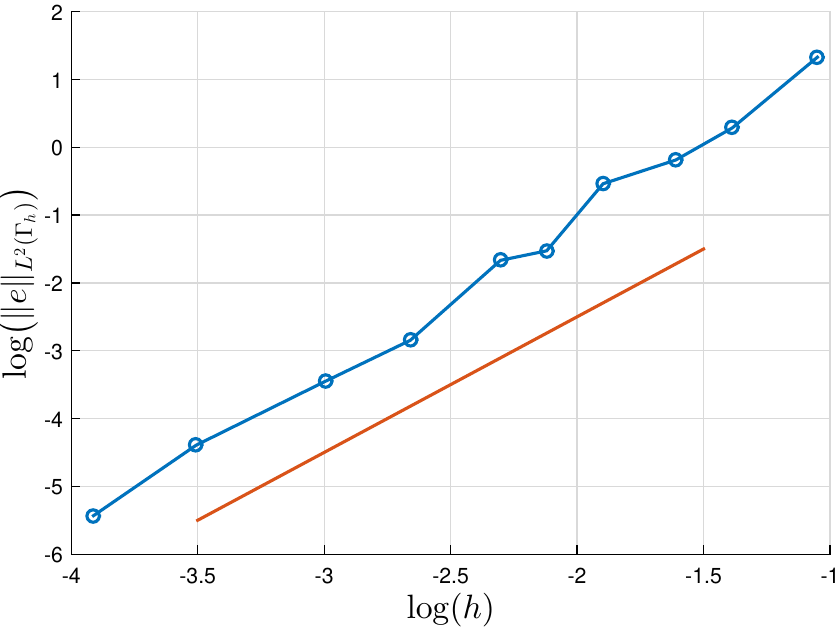}
   \label{fig:torusconv2}
 }
\caption{Convergence studies on unstructured meshes in $L^2(\Gamma_h)$ norm for the two model problems using $\beta=10$. For comparison we include red reference lines with slope~2.}
\label{fig:conv2}
\end{figure}

\subsection{Choice of penalty parameter $\boldsymbol\beta$}
In Figure~\ref{fig:betaconv} we study how the choice of the penalty parameter $\beta$ affects the convergence in the two model problems on unstructured meshes. The numerical study indicates that $\beta=10$ is a good choice. From the analysis we have that we must choose $\beta$ large enough for the error estimate to hold, and we also notice some instability for small $\beta$ in the numerical study. On the other hand, choosing $\beta$ too large will delay the asymptotic regime which is also seen in the numerical study.

By reviewing the analysis we also note that the proof of coercivity for the method (Lemma~\ref{lemma:basics}) does not depend on the local curvature of the surface as it is based on an inverse estimate \eqref{eq:inverseest} on the discrete surface $\Gamma_h$, which is locally flat.
Thus, the choice of $\beta$ does not directly depend on the local curvature of the problem.

\begin{figure}
\centering
\subfigure[Sphere]{
   \includegraphics[width=0.45\linewidth]{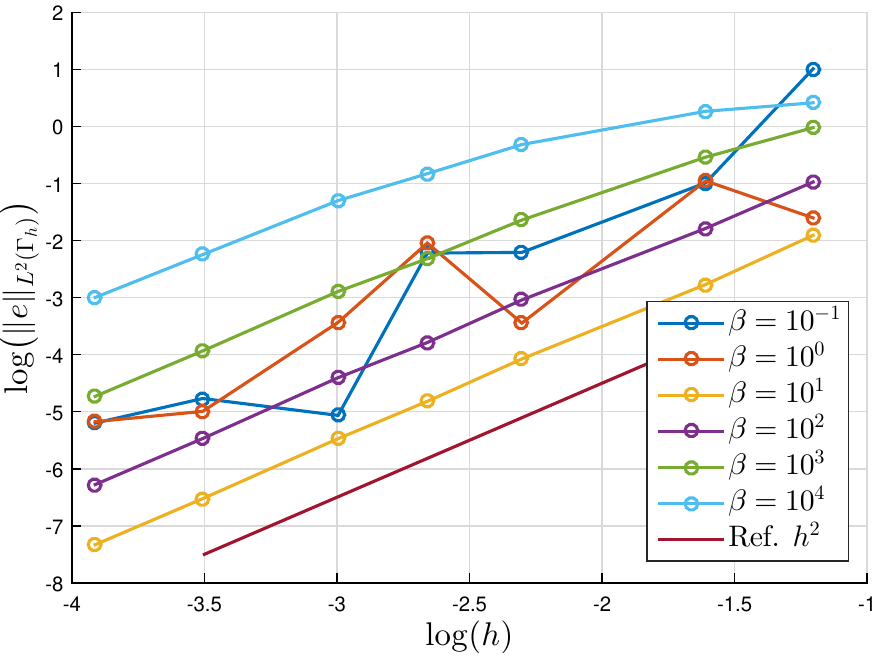}
   \label{fig:sphereconvbeta}
 }
 \subfigure[Torus]{
   \includegraphics[width=0.45\linewidth]{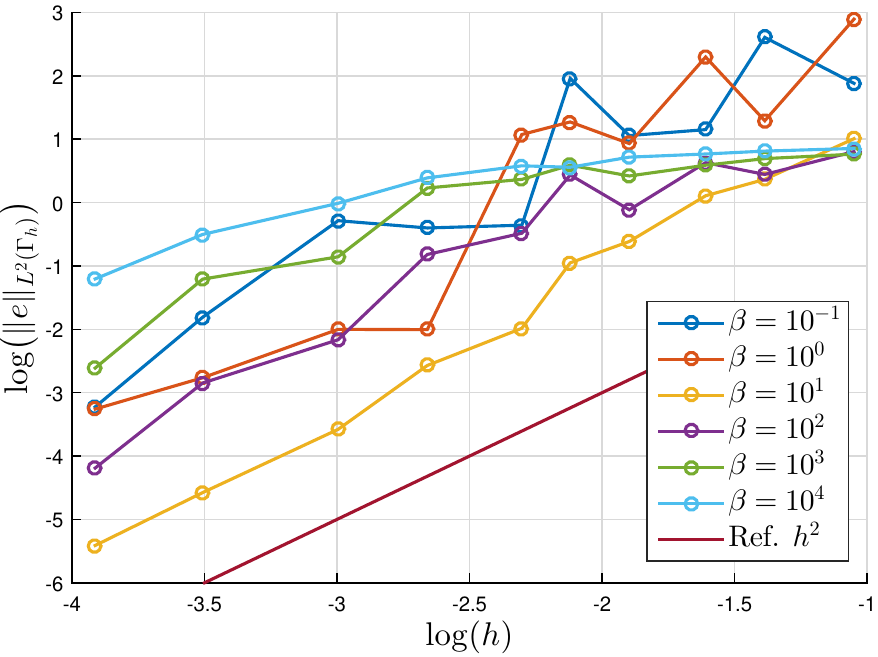}
   \label{fig:torusconvbeta}
 }
\caption{Numerical study on how the choice of $\beta$ affects convergence. Unstructured meshes are used and the error is measured in $L^2(\Gamma_h)$ norm.}
\label{fig:betaconv}
\end{figure}



\appendix

\section{Proof of Lemma \ref{lemma:normalinterpolation}} \label{geomproof}

\begin{proof}
Estimate \eqref{dest} follows from the definition of $d$ and a standard interpolation estimate
\begin{align}
\| d(\bfx) \|_{L^\infty(\hK)} = \| d(\bfx) - \pi_1 d(\bfx) \|_{L^\infty(\hK)} \leq 2 h^2 \| \kappa \|_{L^\infty(\hK)}
\lesssim h^2
\end{align}
where $\pi_1$ is the linear Lagrange interpolation operator on $K$, cf. \cite{claes}.

To prove \eqref{nest} we first note that for $\bfx\in\hK$ we have $\cpm(\bfx) - \bfx = \cpm(\bfx) - \pi_1 \cpm(\bfx)$. By standard interpolation estimates we for any unit vector $\bfb$ in the facet tangent plane have
\begin{align}
\| \cpm_\bfb(\bfx) - \bfb \|_{L^\infty(\hK)}
&=
\| (\bfb\cdot\rgrad) (\cpm - \bfx) \|_{L^\infty(\hK)}
=
\| (\bfb\cdot\rgrad) (\cpm - \pi_1\cpm) \|_{L^\infty(\hK)}
\lesssim h \label{tsomeest}
\end{align}
where $\cpm_\bfb=(\bfb\cdot\rgrad) \cpm$ and the constant in $\lesssim$ depends on $d$ and its derivatives. As
\begin{align}
\cpm_\bfb=(\bfb\cdot\rgrad) \cpm = (\rgrad\otimes\cpm)^T b = (\bP- d H) b
\label{eq:pbexpr}
\end{align}
we note that $| \cpm_\bfb |$ must be bounded from below and above independent of $h$ and we may prove estimate \eqref{nest} by letting $\bfa$ and $\bfb$ be orthogonal unit vectors in the facet tangent plane such that $\hn = \bfa \times \bfb$. Using the boundedness, the triangle inequality, and \eqref{tsomeest} we then get
\begin{align}
\| \bn - &\hn \|_{L^\infty(\hK)}
\lesssim \| \cpm_\bfa \times \cpm_\bfb - \bfa \times \bfb \|_{L^\infty(\hK)}
\\&=
\| (\cpm_\bfa - \bfa) \times (\cpm_\bfb - \bfb)
+ (\cpm_\bfa - \bfa) \times \bfb + \bfa \times (\cpm_\bfb - \bfb) \|_{L^\infty(\hK)}
\lesssim h
\end{align}
Now estimate \eqref{Pnhest} readily follows as
$
\| \hn - (\bn\cdot\hn)\bn \|_{L^\infty(\Gamma_h)}
\leq 2
\| \bn - \hn \|_{L^\infty(\Gamma_h)}
$
and analogously so does estimate \eqref{Phnest}.
Also, by noting that
$1-\bn\cdot\hn=\frac{1}{2}(\bn-\hn)\cdot(\bn-\hn)=\frac{1}{2}|\bn - \hn|^2$,
estimate \eqref{oneest} follows.

Similarly to the proof of \eqref{nest}, we may prove \eqref{test} by letting $\bfb$ be a unit tangent vector to a facet edge $E$. By writing $\bEn\nf$ and $\hEn^\nf$ as the cross product of orthogonal vectors, using the boundedness of $|p_b|$ and the triangle inequality, we have
\begin{align}
&\| \bEn^\nf - \bP \hEn^\nf \|_{L^\infty(E)}
\lesssim
\| \cpm_\bfb \times \bn - \bP (\bfb \times \hn^\nf) \|_{L^\infty(E)}
\\&\qquad\qquad=
\| (\cpm_\bfb - \bP \bfb) \times \bn + \bP\bfb \times \bn - \bP ( \bfb \times \hn^\nf) \|_{L^\infty(E)}
\\&\qquad\qquad\lesssim
\left(
\| \cpm_\bfb - \bP \bfb \|_{L^\infty(E)} + \| \bP\bfb \times \bn - \bP ( \bfb \times \hn^\nf) \|_{L^\infty(E)}
\right)
\end{align}
Due to \eqref{eq:pbexpr} and \eqref{dest} we have $\| \cpm_\bfb - \bP \bfb \|_{L^\infty(E)} \lesssim h^2$.
Note that we may write
\begin{align}
\bP ( \bfb \times \hn^\nf) =   \bP\bfb \times (\bn\otimes\bn)\hn^\nf + (\bn\otimes\bn)\bfb \times \bP\hn^\nf
\end{align}
and thus we have
\begin{align}
\nonumber
&\| \bP\bfb \times \bn - \bP ( \bfb \times \hn^\nf) \|_{L^\infty(E)}
\\&\qquad\lesssim \left(
\| \bP\bfb \times (\bn - (\bn\otimes\bn)\hn^\nf) \|_{L^\infty(E)}
+
\| (\bn\otimes\bn)\bfb \times \bP\hn^\nf \|_{L^\infty(E)}
\right)
\\&\qquad\lesssim \left(
\| 1 - \bn\cdot\hn^\nf \|_{L^\infty(E)}
+
\| p_b -\bfb \|_{L^\infty(E)} \| \bP\hn^\nf \|_{L^\infty(E)}
\right)
\lesssim h^2
\end{align}
where we use that $\bn\cdot p_b = 0$ in the second last inequality and estimates
\eqref{oneest}, \eqref{tsomeest}, and \eqref{Pnhest} in the last inequality.
We have thus shown estimate \eqref{test}
and this concludes the proof.
\qquad
\end{proof}


\section{Tangential derivatives and proof of Lemma \ref{sobolevestimate}} \label{sobolevproof}
Before turning to the actual proof of Lemma~\ref{sobolevestimate} we first give some preliminary results
which will also be used in the proof of Lemma~\ref{lemma:adiff}.

\subsection{Tangential derivatives of extended functions in $n$-mode notation} \label{sect:tanglift}
Using the $n$-mode product \cite[Sect. 2.5]{tensor2009} briefly presented for tensor-matrix multiplications in Section~\ref{sect:surfsobolev} we may also express
tensor-vector multiplications. Componentwise the $n$-mode multiplication between a $k$:th order tensor $T\in\mathbb{R}^{3\times \cdots \times 3}$ and a vector $a\in \mathbb{R}^3$ is defined as
\begin{align}
\left(T \bar{\times}_n a\right)_{i_1 \cdots i_{n-1} i_{n+1} \cdots i_k}
= \sum_{j = 1}^3 T_{i_1 \cdots i_{n-1} j i_{n+1} \cdots i_k} a_{j}
\label{eq:nmodevector}
\end{align}
for $i_1, \cdots, i_{n-1}, i_{n+1}, \cdots ,i_k \in \{1,2,3\}$
which results in a $(k-1)$:th order tensor.
We use the special notation $\bar{\times}_n$ for tensor-vector $n$-mode products as the
order in which tensor-vector products are evaluated matter
\begin{align}
T \bar{\times}_m a \bar{\times}_n b
=
(T \bar{\times}_m a) \bar{\times}_{n-1} b
=
(T \bar{\times}_n b) \bar{\times}_{m} a
\quad \text{for $a,b\in\mathbb{R}^3$ and $m<n$}
\end{align}
while tensor-matrix $n$-mode products ${\times}_n$ are independent of the order
\begin{align}
T {\times}_m A {\times}_n B
=
(T {\times}_m A) {\times}_{n} B
=
(T {\times}_n B) {\times}_{m} A
\quad \text{for $A,B\in\mathbb{R}^{3\times 3}$ and $m\neq n$}
\end{align}
Using this notation we are able to explicitly state the first three
tensors of tangential derivatives in $U$ for extended functions, as presented below.

For an extended tensor $T^\ell$ of order $k$ we by the chain rule have the identity
\begin{align}
T^\ell \otimes \leftgrad
=
(T \circ \cpm) \otimes \leftgrad
=
(T \otimes \leftgrad)\circ \cpm \times_{k+1} (\nabla \otimes \cpm)
=
(T \otimes \leftgrad)^\ell \times_{k+1} B
\label{eq:chainrule}
\end{align}
for $x \in U$, where $B := \nabla \otimes \cpm = \bP(x) - d(x) H(x)$. Note that $B$ is tangential, i.e. $B=PB=BP=PBP$.

By this identity we have that, for $x \in U$, the first order tangential derivatives of an extended function may be expressed as
\begin{align}
D_\Gamma w^\ell &:=
\left\lproj
w^\ell \otimes \leftgrad
\right\rbproj
=
( D_\Gamma w )^\ell \times_1 B = B \, \bar{\times}_2 \, ( D_\Gamma w )^\ell
\label{eq:firstorderderext}
\end{align}
Again applying \eqref{eq:chainrule} gives
\begin{align}
D_\Gamma w^\ell \otimes \leftgrad &=
\left( (D_\Gamma w)^\ell \otimes \leftgrad \right) \times_1 B + \left( B \otimes\leftgrad \right)  \, \bar{\times}_2 \, ( D_\Gamma w )^\ell
\\&=
\left( D_\Gamma w \otimes \leftgrad \right)^\ell \times_1 B \times_2 B + \left(B \otimes\leftgrad \right)  \, \bar{\times}_2 \, ( D_\Gamma w )^\ell 
\\&=
\left( D_\Gamma^2 w \right)^\ell \times_1 B \times_2 B
+
\left( B \otimes\leftgrad \right) \, \bar{\times}_2 \, ( D_\Gamma w )^\ell
\end{align}
and thus we conclude that, for $x\in U$, the tensor of second order tangential derivatives of an extended function can be written
\begin{align}
D_\Gamma^2 w^\ell &:=
\left\lproj
D_\Gamma w^\ell \otimes \leftgrad
\right\rbproj
=
( D_\Gamma^2 w )^\ell \times_1 B \times_2 B
+
D_\Gamma(B) \, \bar{\times}_2 \, ( D_\Gamma w )^\ell
\end{align}
For completeness we also express the tensor of third order tangential derivatives
\begin{align}
D_\Gamma^3 w^\ell &=
( D_\Gamma^3 w )^\ell \times_1 B \times_2 B \times_3 B
\\&\quad+ \nonumber
D_\Gamma(B) \times_1 ( D_\Gamma^2 w )^\ell \times_1 B
+
2 D_\Gamma(B) \times_2 ( D_\Gamma^2 w )^\ell \times_2 B
\\&\quad+ \nonumber
D^2_\Gamma(B) \, \bar{\times}_2 \, ( D_\Gamma w )^\ell
\end{align}

Repeating these calculations for the tangential derivatives on the discrete surface of extended functions gives that the first order derivatives may be expressed
\begin{align}
D_{\Gamma_h} w^\ell &:=
\left\lproj
w^\ell \otimes \leftgrad
\right\rhproj
=
\left\lproj
( D_\Gamma w )^\ell \times_1 B
\right\rhproj
\label{eq:kgujds}
\\&=
( D_\Gamma w )^\ell \times_1 B \times_1 \hP = B \, \bar{\times}_2 \, ( D_\Gamma w )^\ell \times_1 \hP
\label{eq:approxT1}
\end{align}
and the tensor of second order derivatives may be expressed
\begin{align}
D_{\Gamma_h}^2 w^\ell &:=
\left\lproj
D_{\Gamma_h} w^\ell \otimes \leftgrad
\right\rhproj
\\&=
\left\lproj
\left( (D_\Gamma w)^\ell \otimes \leftgrad \right) \times_1 B + \left( B \otimes\leftgrad \right) \, \bar{\times}_2 \, ( D_\Gamma w )^\ell
\right\rhproj
\\&=
\left\lproj
( D_\Gamma w \otimes \leftgrad )^\ell \times_1 B \times_2 B + \left( B \otimes\leftgrad \right)  \, \bar{\times}_2 \, ( D_\Gamma w )^\ell
\right\rhproj
\\&=\label{eq:idusvn2}
\left\lproj
(D^2_\Gamma w )^\ell \times_1 B \times_2 B + \left( B \otimes\leftgrad \right) \, \bar{\times}_2 \, ( D_\Gamma w )^\ell
\right\rhproj
\end{align}
when $x \in U$. As $B=P - dH$ we may rewrite the last term
\begin{align}
&\left\lproj
\left( B \otimes\leftgrad \right) \, \bar{\times}_2 \, ( D_\Gamma w )^\ell
\right\rhproj
= \label{eq:approxTend}
\\&\qquad\quad \nonumber
\left\lproj
- \bn\otimes H( D_\Gamma w )^\ell
- H( D_\Gamma w )^\ell \otimes\bn
-d\left( H \otimes\leftgrad \right) \, \bar{\times}_2 \, ( D_\Gamma w )^\ell
\right\rhproj
\end{align}

\subsection{Proof of Lemma \ref{sobolevestimate}}

\begin{proof}
Let $\mu_h$ be defined as in \S\ref{section:geomapprox}.
By \eqref{integralmeasure2} we for $x\in\Gamma_h$ clearly have $c\geq 1$ such that
$
0 < \frac{1}{c} \leq \mu_h(x) \leq c < \infty
$
and due to how we define our extension to $U$ estimate \eqref{aa1} follows.
Estimate \eqref{aa2} follows from \eqref{aa3} and \eqref{aa4} and thus we turn to proving these two estimates.

\proofterm{Estimate (\ref{aa3}):}
Consider a $k$:th order tensor $T_{\Gamma}$ tangential to $\Gamma$, i.e. $T_{\Gamma} = \left\lproj T_{\Gamma} \right\rbproj$. For the extended tensor $T_{\Gamma}^\ell$  it by the product rule holds that
\begin{align}
\label{eq:v9g8fhnb}
T_{\Gamma}^\ell \otimes \leftgrad
 &=
\lproj T_{\Gamma}^\ell \rbproj \otimes \leftgrad
\\&=
\lproj T_{\Gamma}^\ell \otimes \leftgrad \rbproj
+
T_{\Gamma}^\ell \times \text{[derivatives on projections]}
\\&= \label{eq:ogijsd}
\lproj (T_{\Gamma} \otimes \leftgrad)^\ell \times_{k+1} B \rbproj
+
T_{\Gamma}^\ell \times \text{[derivatives on projections]}
\\&=
\lproj (T_{\Gamma} \otimes \leftgrad)^\ell \rbproj  \times_{k+1} B
+
T_{\Gamma}^\ell \times \text{[derivatives on projections]}
\\&=
\lproj T_{\Gamma} \otimes \leftgrad \rbproj^\ell  \times_{k+1} B
+
T_{\Gamma}^\ell \times \text{[derivatives on projections]}
\\&= \left(D_{\Gamma} T_{\Gamma} \right)^\ell \times_{k+1} B
+
T_{\Gamma}^\ell \times \text{[derivatives on projections]}
\label{eq:h0bds8f9h}
\end{align}
where we in \eqref{eq:ogijsd} use \eqref{eq:chainrule} and in the last equality use \eqref{eq:tensorderivative}.
Note that while we have not formally defined the tensor-tensor multiplication indicated by $\times$ in the above expression, it is sufficient for our purposes to acknowledge that this tensor-tensor multiplication is a linear operation such that the product rule of differentiation holds.
Now assume that
\begin{align}
\label{eq:inductionassumption}
D_{\Gamma_h}^{k} w^\ell = \sum_{m=1}^{k} ( D_\Gamma^m w )^\ell \times T_m
\end{align}
where $T_m$ is a tensor of derivatives (of various orders) on projections and other geometrical quantities (such as $B$), i.e. we assume $D_{\Gamma_h}^{k} w^\ell$ can be expressed as a linear combination of $( D_\Gamma^m w )^\ell$, with $m=1,\dots,k$.
Then by \eqref{eq:Dhk} we have
\begin{align}
\label{eq:kuggksfjn}
D_{\Gamma_h}^{k+1} w^\ell &:=
\lproj D_{\Gamma_h}^{k} w^\ell \otimes \leftgrad \rhproj
\\&=
\left\lproj \Biggl( \sum_{m=1}^{k} ( D_\Gamma^m w )^\ell \times T_m \Biggr) \otimes \leftgrad \right\rhproj
\\&=
\sum_{m=1}^{k} \left\lproj \left(( D_\Gamma^m w )^\ell \otimes \leftgrad \right) \times T_m 
+ ( D_\Gamma^m w )^\ell \times (T_m \otimes \leftgrad)
 \right\rhproj
\label{eq:kuggksfjn2}
\end{align}
Now from calculation \eqref{eq:v9g8fhnb}--\eqref{eq:h0bds8f9h} we can deduce that
$( D_\Gamma^m w )^\ell \otimes \leftgrad$ can be expressed as a linear combination of $( D_\Gamma^{m+1} w )^\ell$ and $( D_\Gamma^m w )^\ell$ and in turn by \eqref{eq:kuggksfjn}--\eqref{eq:kuggksfjn2} we have that $D_{\Gamma_h}^{k+1} w^\ell$
can be expressed as a linear combination of $( D_\Gamma^m w )^\ell$, $m=1,\dots,k+1$.
By \eqref{eq:kgujds}--\eqref{eq:approxT1} the assumption \eqref{eq:inductionassumption} holds for $k=1$
and thus by induction it will hold for any integer $k \geq 1$.

Estimate \eqref{aa3} is now established by taking the (broken) $L^2(\Gamma_h)$ norm on $D_{\Gamma_h}^{k} w^\ell$, writing $D_{\Gamma_h}^{k} w^\ell$ as a linear combination of $( D_\Gamma^m w )^\ell$, where $m=1,\dots,k$, applying the triangle inequality and finally using \eqref{aa1}.

\proofterm{Estimate (\ref{aa4}):}
As established in the proof of estimate \eqref{aa4}, $D_{\Gamma_h}^{k} w^\ell$ may
be expressed as a linear combination of $( D_\Gamma^m w )^\ell$, where $m=1,\dots,k$.
We decompose $D_{\Gamma_h}^{k} w^\ell$ into two terms
$D_{\Gamma_h}^{k} w^\ell = I + II$
where $I$ contain the highest order derivative $( D_\Gamma^k w )^\ell$ and
$II$ contain all lower order derivatives.
This proof will be based on an induction argument and we make the assumption that estimate \eqref{aa4} holds for $k-1$, i.e. that
\begin{align}
| v |_{H^{k-1}_h(\Gamma)} & \lesssim \sum_{m=1}^{k-1} | v |_{H^m_h(\Gamma_h)}
\label{eq:assmjfj}
\end{align}
Clearly, we then for term $II$ have
$\| II \|_{L^2(\Gamma_h)}
\lesssim \sum_{m=1}^{k-1} | v |_{H^m_h(\Gamma)}
\lesssim \sum_{m=1}^{k-1} | v |_{H^m_h(\Gamma_h)}$.
By repeated derivation, as for example in \eqref{eq:kgujds}--\eqref{eq:idusvn2}, we readily see that the term with highest order derivatives can be written
\begin{align}
I &=
\left\lproj
\left( D_{\Gamma}^{k} w \right)^\ell \times_1 B \times_2 \cdots \times_{k} B
\right\rhproj
=
\left( D_{\Gamma}^{k} w \right)^\ell \times_1 \widehat{B} \times_2 \cdots  \times_{k} \widehat{B}
\label{eq:dBBB}
\end{align}
where $\widehat{B} := \hP \left( \bP - d H \right) = \bP - \hn \otimes (\bP\cdot\hn) - d \hP H$. Using this last expression for $\widehat{B}$ and expanding the product we may also decompose $I$ into two terms
\begin{align}
I = I_1 + I_2 = ( D_\Gamma^k w )^\ell + I_2
\end{align}
where we note that each term in $I_2$ must contain $P\cdot \hn$ or $d$ as a factor.
Thus, by Lemma \ref{lemma:normalinterpolation} we can deduce that
\begin{align}
\| I_2 \|_{L^2(\Gamma_h)}
\lesssim h \| ( D_\Gamma^k w )^\ell \|_{L^2(\Gamma_h)}
\lesssim h | w |_{H_h^k(\Gamma)}
\end{align}
where we use estimate \eqref{aa1} in the last inequality.
By estimate \eqref{aa1}, the identity $( D_\Gamma^k w )^\ell = D_{\Gamma_h}^{k} w^\ell  - II - I_2$, and the triangle inequality we then have
\begin{align}
| w |_{H_h^k(\Gamma)}
&\lesssim \| ( D_\Gamma^k w )^\ell \|_{L^2(\Gamma_h)}
\\&
\leq \| D_{\Gamma_h}^{k} w^\ell \|_{L^2(\Gamma_h)}
+ \| II \|_{L^2(\Gamma_h)}
+ \| I_2 \|_{L^2(\Gamma_h)}
\\
&\lesssim | w |_{H_h^k(\Gamma_h)} + \sum_{m=1}^{k-1} | w |_{H_h^m(\Gamma_h)} + h | w |_{H_h^k(\Gamma)}
\end{align}
We can then, under the assumption of \eqref{eq:assmjfj}, establish the estimate using a kick-back argument with the last term. To complete the inductive proof we must establish that the assumption \eqref{eq:assmjfj} holds for $k=2$, or equivalently that estimate \eqref{aa4} holds for $k=1$. In this special case we note that there will be no term $II$ and thus the proof above is complete without using assumption \eqref{eq:assmjfj}. Thus, by induction this completes the proof of estimate \eqref{aa4}.
\end{proof}


\section{Estimates needed in the proof of Lemma~\ref{lemma:adiff}} \label{appendix:Cx}

Reviewing the expressions for the approximate tangential derivatives in Section \ref{sect:tanglift}, i.e. \eqref{eq:approxT1}--\eqref{eq:approxTend},
we deduce that for $x \in \Gamma_h$ we can more explicitly express the first and second order approximate derivatives in terms of the exact operators on $\Gamma$ as
\begin{align}
D_{\Gamma_h} w_h &= \bP_h B(\bfx) (D_\Gamma w_h)^\ell \label{eq:Xa}
\\
D^2_{\Gamma_h} w_h
&=
\left\lproj
(D_\Gamma^2 w_h)^\ell - d \left(H(D_\Gamma^2 w_h)^\ell + (D_\Gamma^2 w_h)^\ell H \right) + d^2 H(D_\Gamma w_h)^\ell H \right\rhproj
\\ \nonumber
&\quad-
\left\lproj
\bn\otimes H(D_\Gamma w_h)^\ell + H(D_\Gamma w_h)^\ell \otimes\bn
+ d\left( H\otimes\leftgrad \right) \bar{\times}_2 (D_\Gamma w_h)^\ell
\right\rhproj
\end{align}
By the trace property
$\trace\left((a\otimes b) A\right) = \trace\left( A (a\otimes b) \right) = a\cdot A^T  b$
for vectors $a,b$ and matrix $A$ we clearly have $\trace\left( \hP B (\hn \otimes\hn \right))=0$
for any $3\times 3$-matrix $B$.
Using this property and the above relations we may express the approximate Laplace--Beltrami operator
$\hlaplace w_h
:=\trace\left( D^2_{\Gamma_h} w_h^\ell \right)$ as
\begin{align}
\hlaplace w_h
&=
\left( \blaplace w_h \right)^\ell - \hn\cdot\left( D^2_{\Gamma} w_h \right)^\ell\cdot\hn
\label{eq:lapexpl1}
\\&\quad \nonumber
- d \left( 2\,\trace\left( H \left( D^2_{\Gamma} w_h \right)^\ell \right)
- 2\hn\cdot\left( H \left( D^2_{\Gamma} w_h \right)^\ell \right)\cdot\hn \right)
\\&\quad \nonumber
+ d^2 \left( \trace\left( H \left( D^2_{\Gamma} w_h \right)^\ell H \right)
- \hn\cdot\left( H \left( D^2_{\Gamma} w_h \right)^\ell H \right)\cdot\hn \right)
\\&\quad \nonumber
+ 2 (\bn\cdot\hn) \hn \cdot H \left(D_\Gamma w_h \right)^\ell
- d \, \trace\left( \hP\left( \left( H\otimes\leftgrad \right) \bar{\times}_2 (D_\Gamma w_h)^\ell \right) \right) 
\end{align}
where we note that all higher order terms except $\left( \blaplace w_h \right)^\ell$ scale with at least $h^2$ while the two lower order terms, i.e. the last line above, scale as $h$ and $h^2$ respectively.


\begin{lemma} \label{lemma:elmopdiff}
For $w \in W$ the following inequalities hold
\begin{align}
\left\|
\left( \blaplace w \right)^\ell
-  \hlaplace w^\ell 
\right\|_{L^2(\Gamma_{h})}
&\lesssim h \| w \|_{H_h^*(\Gamma)} \label{est:lapdiff}
\\
h^{-1/2} \left\|
\left(\bEn \cdot \lJ \bgrad w \rJ \right)^\ell
- \hEn \cdot \lJ \hgrad w^\ell \rJ
\right\|_{L^2(\mathcal{E}_{h})}
&\lesssim h \| w \|_{H_h^*(\Gamma)} \label{est:edgegraddiff}
\\
h^{1/2} \left\|
\left(\left\lA \blaplace w \right\rA\right)^\ell
- \left\lA \hlaplace w^\ell \right\rA
\right\|_{L^2(\mathcal{E}_{h})}
&\lesssim h \| w \|_{H_h^*(\Gamma)} \label{est:edgelapdiff}
\end{align}
\end{lemma}

\begin{proof}

\proofterm{Estimate \ref{est:lapdiff}:}
Reviewing \eqref{eq:lapexpl1} we see that the expression in estimate \eqref{est:lapdiff}
scale with at least $h$ for the first order terms and with at least $h^2$ for the second order terms.
From this and a change of integration domain the estimate directly follows as
\begin{align}
\left\|
\left( \blaplace w \right)^\ell
-  \hlaplace w^\ell 
\right\|_{L^2(\Gamma_{h})}
&\lesssim
h | w |_{H^1(\Gamma)} + h^2 | w |_{H^2(\Gamma)}
\leq h \| w \|_{H_h^*(\Gamma)}
\end{align}
where the last inequality is due to \eqref{technicalnorm}, the definition of $\|\cdot\|_{H_h^*(\Gamma)}$.

\proofterm{Estimate \ref{est:edgegraddiff}:}
By adding and subtracting terms we by the triangle inequality have
\begin{align}
&\left\|
\left(\bEn \cdot \lJ \bgrad w \rJ \right)^\ell
- \hEn \cdot \lJ \hgrad w^\ell \rJ
\right\|_{L^2(\mathcal{E}_{h})}
\nonumber
\\&\qquad\qquad\qquad\qquad\leq
\left\|
\left(\bEn \cdot \lJ \bgrad w \rJ \right)^\ell
- \bEn \cdot \lJ \bgrad w^\ell \rJ
\right\|_{L^2(\mathcal{E}_{h})}
\label{eq:dofbn0}
\\&\qquad\qquad\qquad\qquad\quad+ \nonumber
\left\|
\bEn \cdot \lJ \bgrad w^\ell \rJ
- \hEn \cdot \lJ \hgrad w^\ell \rJ
\right\|_{L^2(\mathcal{E}_{h})}
\end{align}
where the first term concerns how the operator is affected by the change in integration domain,
and the second term concerns the operator approximation.
For the first term we by \eqref{eq:firstorderderext} readily have the estimate
\begin{align}
\left\|
\left(\bEn \cdot \lJ \bgrad w \rJ \right)^\ell
- \bEn \cdot \lJ \bgrad w^\ell \rJ
\right\|_{L^2(\mathcal{E}_{h})}
&=
\left\|
\bEn \cdot (dH) \lJ \left( \bgrad w \right)^\ell \rJ
\right\|_{L^2(\mathcal{E}_{h})}
\\&\lesssim h^2
\left\|
\lJ \left( \bgrad w \right)^\ell \rJ
\right\|_{L^2(\mathcal{E}_{h})}
\end{align}
where we used \eqref{dest}, the bound for $d$, in the inequality.
For the second term in \eqref{eq:dofbn0} we by \eqref{eq:firstorderderext} and \eqref{eq:Xa} have
\begin{align}
&\left\|
\bEn \cdot \lJ \bgrad w^\ell \rJ
- \hEn \cdot \lJ \hgrad w \rJ
\right\|_{L^2(\mathcal{E}_{h})}
\\&\qquad\qquad\qquad =
\left\|
\bEn \cdot \lJ B (\bgrad w)^\ell \rJ
- \hEn \cdot \lJ \hP B (\bgrad w)^\ell \rJ
\right\|_{L^2(\mathcal{E}_{h})}
\\&\qquad\qquad\qquad =
\bigl\|
(\bEnK^+ - \bP\hEnK^+) \cdot B (\bgrad w^+)^\ell
\\&\qquad\qquad\qquad\qquad+ (\bEnK^- - \bP\hEnK^-) \cdot B (\bgrad w^-)^\ell
\bigr\|_{L^2(\mathcal{E}_{h})} \nonumber
\\&\qquad\qquad\qquad \lesssim h^2
\left\|
\lJ \left( \bgrad w \right)^\ell \rJ
\right\|_{L^2(\mathcal{E}_{h})}
\end{align}
where we use \eqref{test} in the inequality. By a change of integration domain, using the triangle inequality on the jump, and applying the trace inequality \eqref{btraceineq} we have
\begin{align}
\left\|
\lJ \left( \bgrad w_h \right)^\ell \rJ
\right\|_{L^2(\mathcal{E}_{h})}
\lesssim
h^{-1/2} | w_h |_{H^1(\Gamma)} + h^{1/2} | w_h |_{H^2(\Gamma)}
\lesssim h^{-1/2} \| w_h \|_{H_h^*(\Gamma)}
\end{align}
where the last inequality is due to \eqref{technicalnorm}, the definition of $\|\cdot\|_{H_h^*(\Gamma)}$. Estimate \eqref{est:edgegraddiff} now readily follows.

\proofterm{Estimate \ref{est:edgelapdiff}:}
By the same arguments as in the proof of estimate \eqref{est:lapdiff} 
in combination with the trace inequality \eqref{btraceineq} we have
\begin{align}
\left\|
\left(\left\lA \blaplace w \right\rA\right)^\ell
- \left\lA \hlaplace w^\ell \right\rA
\right\|_{L^2(\mathcal{E}_{h})}
&\lesssim
h | w |_{H^1(\mathcal{E}_{\Gamma})} + h^2 | w |_{H^2(\mathcal{E}_{\Gamma})}
\\&\lesssim
h \left( h^{-1/2} |w|_{H^1(\Gamma)} + h^{1/2} |w|_{H^2_h(\Gamma)} \right)
\\&\quad+ \nonumber
h^2 \left( h^{-1/2} |w|_{H^2(\Gamma)} + h^{1/2} |w|_{H^3_h(\Gamma)} \right)
\\&
\lesssim h^{1/2} \| w \|_{H_h^*(\Gamma)}
\end{align}
where the last inequality is due to the definition of $\|\cdot\|_{H_h^*(\Gamma)}$.
This concludes the proof.
\end{proof}


\begin{lemma} \label{lemma:specintegrals}
For $\eta,\phi \in V$ the following integral estimates hold
\begin{align}
\left( ( \blaplace \eta )^\ell ,
( \blaplace \phi )^\ell
- \hlaplace \phi^\ell \right)_{\Gamma_{h}}
&\lesssim h^2 \| \eta \|_{H^3(\Gamma)} \| \phi \|_{H^3(\Gamma)}
\label{est:intlap}
\\
\left( (\left\lA \blaplace \eta \right\rA)^\ell ,
(\bEn \cdot \lJ \bgrad \phi \rJ)^\ell
- \hEn \cdot \lJ \hgrad \phi^\ell \rJ \right)_{\mathcal{E}_{h}}
&\lesssim h^2 \| \eta \|_{H^4(\Gamma)} \| \phi \|_{H^4(\Gamma)}
\label{est:intgrad}
\end{align}
\end{lemma}

\begin{proof}
Note that the $L^1(\Gamma_h)$ and $W^1_1(\Gamma_h)$ norms are defined  in \eqref{eq:L1W1norms}.

\proofterm{Estimate \ref{est:intlap}:}
Reviewing the expression for $\hlaplace \phi^\ell$ in \eqref{eq:lapexpl1}
we note that
\begin{align}
\left( ( \blaplace \eta )^\ell ,
( \blaplace \phi )^\ell
- \hlaplace \phi^\ell \right)_{\Gamma_{h}}
&\lesssim
h^2 | \eta |_{H^2(\Gamma)} \left( | \phi |_{H^1(\Gamma)} + | \phi |_{H^2(\Gamma)} \right)
\\&\quad+ \nonumber
\left(
( \blaplace \eta )^\ell ,
(\bn\cdot\hn) \hn \cdot H (\bgrad \phi)^\ell
\right)_{\Gamma_h}
\end{align}
where the remaining integral term by direct application of the bound \eqref{Pnhest} for $\bP\cdot\hn$ would only scale with $h | \eta |_{H^2(\Gamma)} | \phi |_{H^1(\Gamma)}$ which is insufficient.
Instead we make use of the non-standard geometry approximation of Lemma~\ref{lemma:geomnonstandard}
which is applicable as the integrand may be written as the product between $\hP\cdot\bn$ and a sufficiently
regular function $(\blaplace \eta)^\ell H (\bgrad \phi)^\ell$ which is clear by the following calculation
\begin{align}
\nonumber
&\left(
( \blaplace \eta )^\ell ,
(\bn\cdot\hn) \hn \cdot H (\bgrad \phi)^\ell
\right)_{\Gamma_h}
\\&\qquad\qquad=
-\left(
( \blaplace \eta )^\ell ,
(\bn - (\bn\cdot\hn) \hn) \cdot H (\bgrad \phi)^\ell
\right)_{\Gamma_h}
\\&\qquad\qquad=
-\left(
\hP \cdot \bn ,
( \blaplace \eta )^\ell H (\bgrad \phi)^\ell
\right)_{\Gamma_h}
\\&\qquad\qquad\lesssim
h^2 \| (\blaplace \eta)^\ell H (\bgrad \phi)^\ell \|_{W_1^1(\Gamma_h)}
\end{align}
As noted in Section \ref{sect:tanglift} in \eqref{eq:v9g8fhnb}-\eqref{eq:h0bds8f9h} the derivative of a lifted tangential tensor $T_\Gamma^\ell$ may be expressed in terms of the lifted tangential tensor $T_\Gamma^\ell$
and the lifted tangential derivative $(D_\Gamma(T_\Gamma)^\ell$. By this property, the product rule, the boundedness of $H$, and the Cauchy--Schwarz inequality we thus have
\begin{align}
\| (\blaplace \eta)^\ell H (\bgrad \phi)^\ell \|_{W_1^1(\Gamma_h)}
&\lesssim
\| \eta \|_{H^3(\Gamma)} \| \phi \|_{H^2(\Gamma)}
\label{eq:sg9oiuj}
\end{align}
which concludes the proof of the estimate.

\proofterm{Estimate \ref{est:intgrad}:}
First recall the Sobolev embedding $H^2(\Gamma) \hookrightarrow L^\infty(\Gamma)$ \cite[Thm. 2.20]{aubin1982}
which implies, as $\eta,\phi \in V$, that both $\blaplace\eta$ and $\bgrad\phi$ are continuous, i.e. $\left\lA \blaplace \eta \right\rA = \blaplace \eta$ and $\lJ \bgrad \phi \rJ = 0$. Thus we can rewrite the left hand side of \eqref{est:intgrad} as
\begin{align}
\nonumber
( (\left\lA \blaplace \eta \right\rA)^\ell & ,
\bEn \cdot \lJ \bgrad \phi \rJ^\ell
- \hEn \cdot \lJ \hgrad \phi \rJ )_{\mathcal{E}_{h}}
\\
&=
-\left( ( \blaplace \eta )^\ell,
\hEn \cdot \lJ \hP B (\bgrad \phi)^\ell \rJ \right)_{\mathcal{E}_{h}}
\\
&=
\left( ( \blaplace \eta )^\ell,
\hEn \cdot \lJ \hP dH (\bgrad \phi)^\ell \rJ \right)_{\mathcal{E}_{h}}
\\&\quad-\left( ( \blaplace \eta )^\ell,
\hEn \cdot \lJ \hP (\bgrad \phi)^\ell \rJ \right)_{\mathcal{E}_{h}}
= I + II
\end{align}
and we will now handle the resulting two terms separately.

Using \eqref{eq:tproperty} we write term $I$ as
\begin{align}
I &:=
\left( ( \blaplace \eta )^\ell,
\hEn \cdot \lJ \hP dH (\bgrad \phi)^\ell \rJ \right)_{\mathcal{E}_{h}}
\\&=
\left( ( \blaplace \eta )^\ell,
(\hEnK^+ + \hEnK^-) \cdot dH (\bgrad \phi)^\ell \right)_{\mathcal{E}_{h}}
\\&\lesssim
\left( \| \bP \hEnK^+ - \bEnK^+ \|_{L^\infty(\mathcal{E}_{h})}
+
\| \bP \hEnK^- - \bEnK^- \|_{L^\infty(\mathcal{E}_{h})} \right)
\\&\qquad\cdot \nonumber
\| d \|_{L^\infty(\mathcal{E}_{h})}
\| \blaplace \eta \|_{L^2(\mathcal{E}_{\Gamma})}
\| \bgrad \phi \|_{L^2(\mathcal{E}_{\Gamma})}
\\&\lesssim
h^2
\| \eta \|_{H^3({\Gamma})}
\| \phi \|_{H^2({\Gamma})}
\end{align}
where we use
$\bP \hEnK^+ + \bP \hEnK^- = (\bP \hEnK^+ - \bEnK^+) + (\bP \hEnK^- -\bEnK^-)$
and the boundedness of $H$ in the first inequality.
Noting that $\| d \|_{L^\infty(\mathcal{E}_\Gamma)} \leq \| d \|_{L^\infty(\Gamma)}$,
using bounds \eqref{dest} and \eqref{test}, and finally applying the
trace inequality \eqref{btraceineq} yields the last inequality.
For term $II$ we again use \eqref{eq:tproperty} and then the divergence theorem to write
\begin{align}
II &:=
-\left( ( \blaplace \eta )^\ell,
\hEn \cdot \lJ \hP (\bgrad \phi)^\ell \rJ \right)_{\mathcal{E}_{h}}
=
\sum_{K\in\mathcal{K}}
-\left( ( \blaplace \eta )^\ell,
\hEnK \cdot (\bgrad \phi)^\ell \right)_{\partial\hK}
\\&
=
-\sum_{K\in\mathcal{K}}
\left( 
\hgrad \cdot \chi_\Gamma^\ell
, 1 \right)_{\hK}
=
-
\left( 
\trace\left(
(\chi_\Gamma^\ell \otimes \leftgrad) \hP
\right)
, 1 \right)_{\mathcal{K}_h}
\end{align}
Before we turn to the estimation of this term we present the following calculation
\begin{align}
\chi_\Gamma^\ell \otimes \leftgrad
&=
(\chi_\Gamma^\ell \otimes \leftgrad) \times_1 \bP + (\bP^\ell \otimes \leftgrad) \, \bar{\times}_2 \, \chi_\Gamma^\ell
\\&=
\left( D_\Gamma\left( \chi_\Gamma \right) \right)^\ell \times_2 B
+ \left( (\bP \otimes \leftgrad)^\ell \, \bar{\times}_2 \, \chi_\Gamma^\ell \right)  \times_2 B
\\&=
\left( D_\Gamma\left( \chi_\Gamma \right) \right)^\ell
+ \left( (\bP \otimes \leftgrad)^\ell \, \bar{\times}_2 \, \chi_\Gamma^\ell \right)  \times_2 P
\\&\quad \nonumber
- d
\left(
\left( D_\Gamma\left( \chi_\Gamma \right) \right)^\ell \times_2 H
+ \left( (\bP \otimes \leftgrad)^\ell \, \bar{\times}_2 \, \chi_\Gamma^\ell \right)  \times_2 H
\right)
\\&=
\left( \chi_\Gamma \otimes \lbgrad \right)^\ell
- d Z
\end{align}
where
$Z := \left(
\left( D_\Gamma\left( \chi_\Gamma \right) \right)^\ell \times_2 H
+ \left( (\bP \otimes \leftgrad)^\ell \, \bar{\times}_2 \, \chi_\Gamma^\ell \right)  \times_2 H
\right)$
and we note that $\chi_\Gamma^\ell \otimes \leftgrad$ is tangential in the second tensorial dimension, i.e.
$\chi_\Gamma^\ell \otimes \leftgrad = (\chi_\Gamma^\ell \otimes \leftgrad) \bP = \chi_\Gamma^\ell \otimes \lbgrad$.
Thus, the divergence on the approximate surface of $\chi_\Gamma^\ell$ can be written
\begin{align}
\hgrad\cdot\chi_\Gamma^\ell
&=
\trace\left(
(\chi_\Gamma^\ell \otimes \lbgrad) \hP
\right)
=
\trace\left(
(\chi_\Gamma^\ell \otimes \lbgrad)
\right)
-
\trace\left(
(\chi_\Gamma^\ell \otimes \lbgrad) \hn\otimes\hn
\right)
\\&=
\left( \bgrad \cdot \chi_\Gamma \right)^\ell - d \trace\left( Z \right)
-\hn\cdot\left( \chi_\Gamma^\ell \otimes \leftgrad \right)\cdot(\bP\cdot\hn)
\end{align}
Returning to term $II$ we by using the above identity express this term as the following three terms
\begin{align}
II &=
\left(
\left( \bgrad \cdot \chi_\Gamma \right)^\ell , 1 \right)_{\mathcal{K}_h}
-d \left( \trace\left( Z \right) , 1 \right)_{\mathcal{K}_h}
- \left( \hn\cdot\left( \chi_\Gamma^\ell \otimes \leftgrad \right)\cdot(\bP\cdot\hn)
,1 \right)_{\mathcal{K}_h}
\\&= II_1 + II_2 + II_3
\end{align}
For term $II_1$ we by a change of integration and the bound \eqref{muest} for $(1 - \mu_h)$ have
\begin{align}
II_1 &:=
\left( \left( \bgrad \cdot \chi_\Gamma \right)^\ell , 1 \right)_{\mathcal{K}_h}
=
\left( \left( \bgrad \cdot \chi_\Gamma \right)^\ell , 1 - \mu_h \right)_{\mathcal{K}_h}
+ \left( \bgrad \cdot \chi_\Gamma , 1 \right)_{\mathcal{K}_\Gamma} \label{eq:visub}
\\&\lesssim \| 1 - \mu_h \|_{L^\infty(\Gamma_h)} \| \eta \|_{H^3(\Gamma)} \| \phi \|_{H^2(\Gamma)}
\lesssim h^2 \| \eta \|_{H^3(\Gamma)} \| \phi \|_{H^2(\Gamma)}
\end{align}
where the last term in \eqref{eq:visub} is zero which follows from the divergence theorem on each curved triangle
\begin{align}
\left( \bgrad \cdot \chi_\Gamma , 1 \right)_{\mathcal{K}_\Gamma}
=
\sum_{K\in\mathcal{K}}
\left( \bEnK \cdot \chi_\Gamma , 1 \right)_{\partial\bK}
=
\sum_{E\in\mathcal{E}}
\left( \bEn \cdot \lJ \bgrad\phi \rJ \blaplace\eta , 1 \right)_{\bE} = 0
\end{align}
and $\lJ \bgrad\phi \rJ = 0$.
As term $II_2$ is multiplied by the distance function $d$ we have the following estimate
\begin{align}
II_2 &:=
-
d \left( \trace\left( Z \right) , 1 \right)_{\mathcal{K}_h}
\lesssim \| d \|_{L^\infty(\Gamma_h)} \| Z \|_{L^1(\Gamma_h)}
\lesssim h^2 \| \eta \|_{H^3(\Gamma)} \| \phi \|_{H^2(\Gamma)}
\end{align}
where we use the bound \eqref{dest} for $d$ and the Cauchy--Schwarz inequality.
For term $II_3$ we now add and subtract terms to get an expression suitable for Lemma~\ref{lemma:geomnonstandard}, such that
\begin{align}
II_3 &:=
-
\left( \hn\cdot\left( \chi_\Gamma^\ell \otimes \leftgrad \right)\cdot(\bP\cdot\hn)
,1 \right)_{\mathcal{K}_\Gamma}
\\&= \label{eq:sdbuhi}
-
\left( (\hn-\bn)\cdot \left( \chi_\Gamma^\ell \otimes \leftgrad \right) , \bP\cdot\hn
\right)_{\mathcal{K}_h}
-
\left( \bn\cdot \left( \chi_\Gamma^\ell \otimes \leftgrad \right) , \bP\cdot\hn
\right)_{\mathcal{K}_h}
\\&\lesssim
\| \bn - \hn \|_{L^\infty(\Gamma_h)} \| \bP\cdot\hn \|_{L^\infty(\Gamma_h)}
\| \chi_\Gamma^\ell \otimes \leftgrad \|_{L^1(\Gamma_h)}
\\&\quad \nonumber
-
\left( \bn\cdot \left( \chi_\Gamma^\ell \otimes \leftgrad \right) , (1 - \bn \cdot \hn) (\bn + \hn)
\right)_{\mathcal{K}_h}
\\&\quad \nonumber
+
\left( \bn\cdot \left( \chi_\Gamma^\ell \otimes \leftgrad \right) , \hP\cdot\bn
\right)_{\mathcal{K}_h} 
\\&\lesssim \label{eq:nbd98fh} h^2
\left(
\| \chi_\Gamma^\ell \otimes \leftgrad \|_{L^1(\Gamma_h)}
+
\| \bn\cdot\left( \chi_\Gamma^\ell \otimes \leftgrad \right) \|_{W^1_1(\Gamma_h)}
\right)
\\&\lesssim h^2
\| \eta \|_{H^4(\Gamma)} \| \phi \|_{H^3(\Gamma)}
\end{align}
where we rewrite the last term in \eqref{eq:sdbuhi} using the identity
$\bP\cdot\hn = (1 - \bn \cdot \hn) (\bn + \hn) - \hP \cdot \bn$
and the first inequality is due to bounds \eqref{nest} and \eqref{Pnhest}.
The second inequality then follows from the bound \eqref{muest} on $1-\mu_h$ and
Lemma~\ref{lemma:geomnonstandard} applied to the last term in \eqref{eq:nbd98fh}.
Finally, the last inequality follows by the same motivation as \eqref{eq:sg9oiuj}.
This concludes the proof of the estimate and the lemma.
\end{proof}


\section{Matlab code} \label{app:matlab}
Below we provide {\sc Matlab} code for the load density function in the model problem on the torus.

{
\footnotesize
\begin{verbatim}
function f=loadfcn(r,phi,th)
f =(9*r^4*sin(2*phi + th) + 491*r^4*sin(4*phi + th) + 324*R^4*sin(2*phi - 3*th) + ...
324*R^4*sin(4*phi + 3*th) + 179*r^4*sin(2*phi - th) + 313*r^4*sin(2*phi - 3*th) + ...
9*r^4*sin(4*phi - th) + 179*r^4*sin(2*phi - 5*th) + 1561*r^4*sin(4*phi + 3*th) + ...
36*r^4*sin(2*phi - 7*th) + 347*r^4*sin(4*phi + 5*th) + 36*r^4*sin(4*phi + 7*th) + ...
366*R^2*r^2*sin(2*phi - th) + 1386*R^2*r^2*sin(2*phi - 3*th) + ...
696*R^2*r^2*sin(2*phi - 5*th) + 2250*R^2*r^2*sin(4*phi + 3*th) + ...
696*R^2*r^2*sin(4*phi + 5*th) + 99*R*r^3*sin(2*phi) + 821*R*r^3*sin(2*phi - 2*th) + ...
570*R^3*r*sin(2*phi - 2*th) + 875*R*r^3*sin(2*phi - 4*th) + ...
1781*R*r^3*sin(4*phi + 2*th) + 798*R^3*r*sin(2*phi - 4*th) + ...
570*R^3*r*sin(4*phi + 2*th) + 261*R*r^3*sin(2*phi - 6*th) + ...
1547*R*r^3*sin(4*phi + 4*th) + 798*R^3*r*sin(4*phi + 4*th) + ...
261*R*r^3*sin(4*phi + 6*th) + 366*R^2*r^2*sin(4*phi + th) + ...
198*R*r^3*cos(2*phi)*sin(2*phi))/(8*R^4*r^4 + 32*R^3*r^5*cos(th) + ...
48*R^2*r^6*cos(th)^2 + 32*R*r^7*cos(th)^3 + 8*r^8*cos(th)^4);
end
\end{verbatim}\vspace{-3pt}
}



\bibliographystyle{abbrv}
\bibliography{BiharmonicRef_v2}
\end{document}